\documentclass{article}%
\usepackage{amsfonts}
\usepackage{graphicx}
\usepackage{amsmath}
\usepackage{amssymb}%
\providecommand{\U}[1]{\protect\rule{.1in}{.1in}}
\newtheorem{theorem}{Theorem}

\newtheorem{corollary}[theorem]{Corollary}

\newtheorem{definition}[theorem]{Definition}

\newtheorem{lemma}[theorem]{Lemma}
\newtheorem{notation}[theorem]{Notation}

\newtheorem{proposition}[theorem]{Proposition}
\newtheorem{remark}[theorem]{Remark}

\newenvironment{proof}[1][Proof]{\textbf{#1.} }{\ \rule{0.5em}{0.5em}}
\begin{document}

\title{Differential Geometry of Microlinear Fr\"{o}licher Spaces IV-2}
\author{Hirokazu Nishimura\\Institute of Mathematics\\University of Tsukuba\\Tsukuba, Ibaraki 305-8571\\Japan}
\maketitle

\begin{abstract}
This paper is the sequel to our previous paper (Differetial Geometry of
Microlinear Fr\"{o}licher spaces IV-1), where three approaches to jet bundles
are presented and compared. The first objective in this paper is to give the
affine bundle theorem for the second and third approaches to jet bundles. The
second objective is to deal with the three approaches to jet bundles in the
context where coordinates are available. In this context all the three
approaches are shown to be equivalent.

\end{abstract}

\section{\label{s0}Introduction}

The principal objectives in this second part of the paper are firstly to deal
with the affine bundle theorem in the second and third approaches to jet
bundles and secondly to treat the three approaches to jet bundles in
\cite{nishimura} within the context where coordinates are available, namely,
$E=\mathbb{R}^{p+q}$, $M=\mathbb{R}^{p}$, and $\pi$ is the canonical
projection. \S \ref{s8} is devoted to the affine bundle theorem. We let
$i$\ ($j$, resp.) range over the natural numbers between $1$\ and
$p$\ (between $1 $ and $q$, resp.), including the endpoints. It is shown that,
within this traditional context, the three approaches are essentially
equivalent. The traditional coordinate approach to jet bundles is given a
noble description after the manner of \cite{kock1} in \S \ref{s3}, where the
affine bundle theorem is established on these lines. Our three approaches are
related to this traditional approach in \S \ref{s4}, \S \ref{s5} and
\S \ref{s6} in order. In particular, the affine bundles in the second and
third approaches are shown to be isomorphic to that in the traditional approach.

\section{\label{s1}Previous Results}

We collect a few results of our previous paper \cite{nishimura}\ to be quoted
in this paper.

\begin{definition}
\label{d4.1}Let $n$ be a natural number. A $D^{n}$-pseudotangential
\textit{over }the bundle $\pi:E\rightarrow M$ \textit{at} $x\in E$ is a
mapping $\nabla_{x}:\left(  M\otimes\mathcal{W}_{D^{n}}\right)  _{\pi
(x)}\rightarrow\left(  E\otimes\mathcal{W}_{D^{n}}\right)  _{x}\ $abiding by
the following conditions:

\begin{enumerate}
\item We have
\[
\left(  \pi\otimes\mathrm{id}_{\mathcal{W}_{D^{n}}}\right)  \left(  \nabla
_{x}(\gamma)\right)  =\gamma
\]
for any $\gamma\in\left(  M\otimes\mathcal{W}_{D^{n}}\right)  _{\pi(x)}$.

\item We have
\[
\nabla_{x}(\alpha\underset{i}{\cdot}\gamma)=\alpha\underset{i}{\cdot}%
\nabla_{x}(\gamma)\quad\quad(1\leq i\leq n)
\]
for any $\gamma\in\left(  M\otimes\mathcal{W}_{D^{n}}\right)  _{\pi(x)}$ and
any $\alpha\in\mathbb{R}$.

\item The diagram
\[%
\begin{array}
[c]{ccc}%
\left(  M\otimes\mathcal{W}_{D^{n}}\right)  _{\pi\left(  x\right)  } &
\rightarrow & \left(  M\otimes\mathcal{W}_{D^{n}}\right)  _{\pi\left(
x\right)  }\otimes\mathcal{W}_{D_{m}}\\
\nabla_{x}\downarrow &  & \downarrow\nabla_{x}\otimes\mathrm{id}%
_{\mathcal{W}_{D_{m}}}\\
\left(  E\otimes\mathcal{W}_{D^{n}}\right)  _{x} & \rightarrow & \left(
E\otimes\mathcal{W}_{D^{n}}\right)  _{x}\otimes\mathcal{W}_{D_{m}}%
\end{array}
\]
is commutative, where $m$\ is an arbitrary natural number, the upper
horizontal arrow is
\[
\mathrm{id}_{M}\otimes\mathcal{W}_{\left(  d_{1},...,d_{n},e\right)  \in
D^{n}\times D_{m}\mapsto\left(  d_{1},...,d_{i-1},ed_{i},d_{i+1}%
,...d_{n}\right)  \in D^{n}}\text{,}%
\]
and the lower horizontal arrow is
\[
\mathrm{id}_{E}\otimes\mathcal{W}_{\left(  d_{1},...,d_{n},e\right)  \in
D^{n}\times D_{m}\mapsto\left(  d_{1},...,d_{i-1},ed_{i},d_{i+1}%
,...d_{n}\right)  \in D^{n}}\text{.}%
\]

\item We have
\[
\nabla_{x}(\gamma^{\sigma})=(\nabla_{x}(\gamma))^{\sigma}%
\]
for any $\gamma\in\left(  M\otimes\mathcal{W}_{D^{n}}\right)  _{\pi(x)}$ and
for any $\sigma\in\mathbf{S}_{n}$.
\end{enumerate}
\end{definition}

\begin{definition}
\label{d4.2}The notion of a $D^{n}$-tangential \textit{over} the bundle
$\pi:E\rightarrow M$ \textit{at} $x$ is defined by induction on $n$. The
notion of a $D$-tangential \textit{over} the bundle $\pi:E\rightarrow M$
\textit{at} $x$ shall be identical with that of a $D$-pseudotangential
\textit{over} the bundle $\pi:E\rightarrow M$ \textit{at} $x$ . Now we proceed
inductively. A $D^{n+1}$-pseudotangential
\[
\nabla_{x}:\left(  M\otimes\mathcal{W}_{D^{n+1}}\right)  _{\pi(x)}%
\rightarrow\left(  E\otimes\mathcal{W}_{D^{n+1}}\right)  _{x}%
\]
$\ $ over the bundle $\pi:E\rightarrow M$ at $x\in E$ is called a $D^{n+1}%
$-tangential \textit{over} the bundle $\pi:E\rightarrow M$ \textit{at} $x$ if
it acquiesces in the following two conditions:

\begin{enumerate}
\item $\widehat{\pi}_{n+1,n}(\nabla_{x})$ is a $D^{n}$-tangential
\textit{over} the bundle $\pi:E\rightarrow M$ \textit{at} $x$.

\item For any $\gamma\in\left(  M\otimes\mathcal{W}_{D^{n}}\right)  _{\pi(x)}%
$, we have
\begin{align*}
&  \nabla_{x}\left(  \left(  \mathrm{id}_{M}\otimes\mathcal{W}_{\left(
d_{1},...,d_{n},d_{n+1}\right)  \in D^{n+1}\mapsto\left(  d_{1},...,d_{n}%
d_{n+1}\right)  \in D^{n}}\right)  \left(  \gamma\right)  \right) \\
&  =\left(  \mathrm{id}_{E}\otimes\mathcal{W}_{\left(  d_{1},...,d_{n}%
,d_{n+1}\right)  \in D^{n+1}\mapsto\left(  d_{1},...,d_{n}d_{n+1}\right)  \in
D^{n+1}}\right)  \left(  \left(  \widehat{\pi}_{n+1,n}(\nabla_{x})\right)
(\gamma)\right)
\end{align*}

\end{enumerate}
\end{definition}

\begin{proposition}
\label{t4.1.4}Let $m,n$ be natural numbers with $m\leq n$. Let $k_{1}%
,...,k_{m}$ be positive integers with $k_{1}+...+k_{m}=n$. For any $\nabla
_{x}\in\mathbb{J}^{D^{n}}(\pi)$, any $\gamma\in\left(  M\otimes\mathcal{W}%
_{D^{m}}\right)  _{\pi(x)}$ and any $\sigma\in\mathbf{S}_{n}$, we have
\begin{align*}
&  \nabla_{x}\left(  \left(  \mathrm{id}_{M}\otimes\mathcal{W}_{(d_{1}%
,...,d_{n})\in D^{n}\mapsto\left(  d_{\sigma(1)}...d_{\sigma(k_{1})}%
,d_{\sigma(k_{1}+1)}...d_{\sigma(k_{1}+k_{2})},...,d_{\sigma(k_{1}%
+...+k_{m-1}+1)}...d_{\sigma(n)}\right)  }\right)  \left(  \gamma\right)
\right) \\
&  =\left(  \mathrm{id}_{E}\otimes\mathcal{W}_{(d_{1},...,d_{n})\in
D^{n}\mapsto\left(  d_{\sigma(1)}...d_{\sigma(k_{1})},d_{\sigma(k_{1}%
+1)}...d_{\sigma(k_{1}+k_{2})},...,d_{\sigma(k_{1}+...+k_{m-1}+1)}%
...d_{\sigma(n)}\right)  }\right)  \left(  \left(  \pi_{n,m}(\nabla
_{x})\right)  (\gamma)\right)
\end{align*}

\end{proposition}

\begin{proposition}
\label{t7.1.4}The diagram
\[%
\begin{array}
[c]{ccc}%
\mathbb{\hat{J}}_{x}^{D^{n+1}}(\pi) & \underrightarrow{\widehat{\psi}_{n+1}} &
\mathbb{\hat{J}}_{x}^{D_{n+1}}(\pi)\\%
\begin{array}
[c]{cc}%
\widehat{\mathbf{\pi}}_{n+1,n} & \downarrow
\end{array}
&  &
\begin{array}
[c]{cc}%
\downarrow & \widehat{\mathbf{\pi}}_{n+1,n}%
\end{array}
\\
\mathbb{\hat{J}}_{x}^{D^{n}}(\pi) & \overrightarrow{\widehat{\psi}_{n}} &
\mathbb{\hat{J}}_{x}^{D_{n}}(\pi)
\end{array}
\]
commutes.
\end{proposition}

\section{\label{s8}The Affine Bundle Theorem}

\subsection{\label{s8.2}The Theorem in the Second Approach}

\subsubsection{\label{s8.2.1}Affine Bundles}

\begin{lemma}
\label{t8.2.1.1}The diagram
\[%
\begin{array}
[c]{ccc}%
D\left\{  n\right\}  _{n-1} & \underrightarrow{i_{D\left\{  n\right\}
_{n-1}\rightarrow D^{n}}} & D^{n}\\
i_{D\left\{  n\right\}  _{n-1}\rightarrow D^{n}}\downarrow &  & \downarrow
\Psi_{D^{n}}\\
D^{n} & \underrightarrow{\Phi_{D^{n}}} & D^{n}\oplus D
\end{array}
\]
is a quasi-colimit diagram, where $i_{D\left\{  n\right\}  _{n-1}\rightarrow
D^{n}}$ is the canonical injection of $D\left\{  n\right\}  _{n-1}$ into
$D^{n} $, and
\begin{align*}
\Phi_{D^{n}}(d_{1},...,d_{n})  &  =(d_{1},...,d_{n},0)\\
\Psi_{D^{n}}(d_{1},...,d_{n})  &  =(d_{1},...,d_{n},d_{1}...d_{n})\text{.}%
\end{align*}

\end{lemma}

This implies directly that

\begin{proposition}
\label{t8.2.1.2}Given $\gamma_{+},\gamma_{-}\in$ $M\otimes\mathcal{W}_{D^{n}}$
with
\[
\left(  \mathrm{id}_{M}\otimes\mathcal{W}_{i_{D\left\{  n\right\}
_{n-1}\rightarrow D^{n}}}\right)  \left(  \gamma_{+}\right)  =\left(
\mathrm{id}_{M}\otimes\mathcal{W}_{i_{D\left\{  n\right\}  _{n-1}\rightarrow
D^{n}}}\right)  \left(  \gamma_{-}\right)  \text{,}%
\]
there exists unique $\gamma\in M\otimes\mathcal{W}_{D^{n}\oplus D}$ with
\begin{align*}
\left(  \mathrm{id}_{M}\otimes\mathcal{W}_{\Psi_{D^{n}}}\right)  \left(
\gamma\right)   &  =\gamma_{+}\text{ and}\\
\left(  \mathrm{id}_{M}\otimes\mathcal{W}_{\Phi_{D^{n}}}\right)  \left(
\gamma\right)   &  =\gamma_{-}%
\end{align*}

\end{proposition}

\begin{notation}
Under the same notation as in the above proposition, we denote
\[
\left(  \mathrm{id}_{M}\otimes\mathcal{W}_{\Xi_{D^{n}}}\right)  \left(
\gamma\right)
\]
by $\gamma_{+}\dot{-}\gamma_{-}$, where $\Xi_{D^{n}}:D\rightarrow D^{n}\oplus
D$ is the mapping
\[
d\in D\mapsto\left(  0,...,0,d\right)  \in D^{n}\oplus D
\]

\end{notation}

From the very definition of $\dot{-}$, we have

\begin{proposition}
\label{t8.2.1.3}Let $F$ be a mapping of $M$\ into $M^{\prime}$. Given
$\gamma_{+},\gamma_{-}\in M\otimes\mathcal{W}_{D^{n}}$ with
\[
\left(  \mathrm{id}_{M}\otimes\mathcal{W}_{i_{D\left\{  n\right\}
_{n-1}\rightarrow D^{n}}}\right)  \left(  \gamma_{+}\right)  =\left(
\mathrm{id}_{M}\otimes\mathcal{W}_{i_{D\left\{  n\right\}  _{n-1}\rightarrow
D^{n}}}\right)  \left(  \gamma_{-}\right)  \text{,}%
\]
we have
\[
\left(  \mathrm{id}_{M^{\prime}}\otimes\mathcal{W}_{i_{D\left\{  n\right\}
_{n-1}\rightarrow D^{n}}}\right)  \left(  \left(  F\otimes\mathrm{id}%
_{\mathcal{W}_{D^{n}}}\right)  \left(  \gamma_{+}\right)  \right)  =\left(
\mathrm{id}_{M^{\prime}}\otimes\mathcal{W}_{i_{D\left\{  n\right\}
_{n-1}\rightarrow D^{n}}}\right)  \left(  \left(  F\otimes\mathrm{id}%
_{\mathcal{W}_{D^{n}}}\right)  \left(  \gamma_{-}\right)  \right)
\]
and
\begin{align*}
&  \left(  F\otimes\mathrm{id}_{\mathcal{W}_{D}}\right)  \left(  \gamma
_{+}\dot{-}\gamma_{-}\right) \\
&  =\left(  F\otimes\mathrm{id}_{\mathcal{W}_{D^{n}}}\right)  \left(
\gamma_{+}\right)  \dot{-}\left(  F\otimes\mathrm{id}_{\mathcal{W}_{D^{n}}%
}\right)  \left(  \gamma_{-}\right)
\end{align*}

\end{proposition}

\begin{lemma}
\label{t8.2.1.4}The diagram
\[
\
\begin{array}
[c]{ccc}%
1 & \underrightarrow{i_{1\rightarrow D}} & D\\
i_{1\rightarrow D^{n}}\downarrow &  & \downarrow\Xi_{D^{n}}\\
D^{n} & \underrightarrow{\Phi_{D^{n}}} & D^{n}\oplus D
\end{array}
\]
is a quasi-colimit diagram, where $i_{1\rightarrow D^{n}}$ is the canonical
injection of $1$\ into $D^{n}$ and $i_{1\rightarrow D}$ is the canonical
injection of $1$\ into $D$.
\end{lemma}

This implies directly that

\begin{proposition}
\label{t8.2.1.5}Given $t\in M\otimes\mathcal{W}_{D}\ $and $\gamma\in
M\otimes\mathcal{W}_{D^{n}}\ $with
\[
\left(  \mathrm{id}_{M}\otimes\mathcal{W}_{i_{1\rightarrow D}}\right)  \left(
t\right)  =\left(  \mathrm{id}_{M}\otimes\mathcal{W}_{i_{1\rightarrow D^{n}}%
}\right)  \left(  \gamma\right)  \text{,}%
\]
there exists unique $\gamma^{\prime}\in M\otimes\mathcal{W}_{D^{n}\oplus D}%
\ $with
\begin{align*}
\left(  \mathrm{id}_{M}\otimes\mathcal{W}_{\Xi_{D^{n}}}\right)  \left(
\gamma^{\prime}\right)   &  =t\text{ and}\\
\left(  \mathrm{id}_{M}\otimes\mathcal{W}_{\Phi_{D^{n}}}\right)  \left(
\gamma^{\prime}\right)   &  =\gamma\text{.}%
\end{align*}

\end{proposition}

\begin{notation}
Under the same notation as in the above proposition, we denote
\[
\left(  \mathrm{id}_{M}\otimes\mathcal{W}_{\Psi_{D^{n}}}\right)  \left(
\gamma^{\prime}\right)
\]
by $t\dot{+}\gamma$, where $\Psi_{D^{n}}$ is as in Lemma \ref{t8.2.1.1}
\end{notation}

From the very definition of $\dot{+}$,$\ $we have

\begin{proposition}
\label{t8.2.1.6}Let $F$ be a mapping of $M$\ into $M^{\prime}$. Given $t\in
M\otimes\mathcal{W}_{D}\ $and $\gamma\in M\otimes\mathcal{W}_{D^{n}}\ $with
\[
\left(  \mathrm{id}_{M}\otimes\mathcal{W}_{i_{1\rightarrow D}}\right)  \left(
t\right)  =\left(  \mathrm{id}_{M}\otimes\mathcal{W}_{i_{1\rightarrow D^{n}}%
}\right)  \left(  \gamma\right)  \text{,}%
\]
we have
\[
\left(  \mathrm{id}_{M^{\prime}}\otimes\mathcal{W}_{i_{1\rightarrow D}%
}\right)  \left(  \left(  F\otimes\mathrm{id}_{\mathcal{W}_{D}}\right)
\left(  t\right)  \right)  =\left(  \mathrm{id}_{M^{\prime}}\otimes
\mathcal{W}_{i_{1\rightarrow D^{n}}}\right)  \left(  \left(  F\otimes
\mathrm{id}_{\mathcal{W}_{D^{n}}}\right)  \left(  \gamma\right)  \right)
\]
and
\[
\left(  F\otimes\mathrm{id}_{\mathcal{W}_{D^{n}}}\right)  \left(  t\dot
{+}\gamma\right)  =\left(  F\otimes\mathrm{id}_{\mathcal{W}_{D}}\right)
\left(  t\right)  \dot{+}\left(  F\otimes\mathrm{id}_{\mathcal{W}_{D^{n}}%
}\right)  \left(  \gamma\right)  \text{.}%
\]

\end{proposition}

We can proceed as in \S \S 3.4 of \cite{la} to get

\begin{theorem}
\label{t8.2.1.7}The canonical projection $\mathrm{id}_{M}\otimes
\mathcal{W}_{i_{D\left\{  n\right\}  _{n-1}\rightarrow D^{n}}}:M\otimes
\mathcal{W}_{D^{n}}\mathcal{\rightarrow}M\otimes\mathcal{W}_{D\left\{
n\right\}  _{n-1}}\ $is an affine bundle over the vector bundle $\left(
M\otimes\mathcal{W}_{D}\right)  \underset{M}{\times}\left(  M\otimes
\mathcal{W}_{D\left\{  n\right\}  _{n-1}}\right)  \rightarrow M\otimes
\mathcal{W}_{D\left\{  n\right\}  _{n-1}}$.
\end{theorem}

We have the following $n$-dimensional$\ $counterparts of Propositions 5, 6 and
7 in \S \S 3.4 of \cite{la}.

\begin{proposition}
\label{t8.2.1.8}For any $\alpha\in\mathbb{R}$, any $\gamma_{+},\gamma
_{-},\gamma\in M\otimes\mathcal{W}_{D^{n}}\ $and any $t\in M\otimes
\mathcal{W}_{D}\ $with
\[
\left(  \mathrm{id}_{M}\otimes\mathcal{W}_{i_{D\left\{  n\right\}
_{n-1}\rightarrow D^{n}}}\right)  \left(  \gamma_{+}\right)  =\left(
\mathrm{id}_{M}\otimes\mathcal{W}_{i_{D\left\{  n\right\}  _{n-1}\rightarrow
D^{n}}}\right)  \left(  \gamma_{-}\right)
\]
and
\[
\left(  \mathrm{id}_{M}\otimes\mathcal{W}_{i_{1\rightarrow D}}\right)  \left(
t\right)  =\left(  \mathrm{id}_{M}\otimes\mathcal{W}_{i_{1\rightarrow D^{n}}%
}\right)  \left(  \gamma\right)  \text{,}%
\]
we have
\begin{align*}
\alpha(\gamma_{+}\dot{-}\gamma_{-})  &  =(\alpha\underset{i}{\cdot}\gamma
_{+})\dot{-}(\alpha\underset{i}{\cdot}\gamma_{-})\\
\alpha\underset{i}{\cdot}(t\dot{+}\gamma)  &  =\alpha t\dot{+}\alpha
\underset{i}{\cdot}\gamma
\end{align*}

\end{proposition}

\begin{proposition}
\label{t8.2.1.9}The diagrams
\[%
\begin{array}
[c]{ccc}%
\left(  M\otimes\mathcal{W}_{D^{n}}\right)  \underset{M\otimes\mathcal{W}%
_{D\left\{  n\right\}  _{n-1}}}{\times}\left(  M\otimes\mathcal{W}_{D^{n}%
}\right)  & \rightarrow & M\otimes\mathcal{W}_{D}\\
\downarrow_{i} &  & \downarrow\\
\left(  \left(  M\otimes\mathcal{W}_{D^{n}}\right)  \underset{M\otimes
\mathcal{W}_{D\left\{  n\right\}  _{n-1}}}{\times}\left(  M\otimes
\mathcal{W}_{D^{n}}\right)  \right)  \otimes\mathcal{W}_{D_{m}} & \rightarrow
& \left(  M\otimes\mathcal{W}_{D}\right)  \otimes\mathcal{W}_{D_{m}}%
\end{array}
\left(  1\leq i\leq n\right)
\]
\[%
\begin{array}
[c]{ccc}%
\left(  M\otimes\mathcal{W}_{D}\right)  \underset{M}{\times}\left(
M\otimes\mathcal{W}_{D^{n}}\right)  & \rightarrow & M\otimes\mathcal{W}%
_{D^{n}}\\
\downarrow_{i} &  & \downarrow_{i}\\
\left(  \left(  M\otimes\mathcal{W}_{D}\right)  \underset{M}{\times}\left(
M\otimes\mathcal{W}_{D^{n}}\right)  \right)  \otimes\mathcal{W}_{D_{m}} &
\rightarrow & \left(  M\otimes\mathcal{W}_{D^{n}}\right)  \otimes
\mathcal{W}_{D_{m}}%
\end{array}
\left(  1\leq i\leq n\right)
\]
are commutative, where

\begin{enumerate}
\item In the former diagram, the lower horizontal arrow represents
\begin{align*}
&  \left(  \left(  \gamma_{+},\gamma_{-}\right)  \in\left(  M\otimes
\mathcal{W}_{D^{n}}\right)  \underset{M\otimes\mathcal{W}_{D\left\{
n\right\}  _{n-1}}}{\times}\left(  M\otimes\mathcal{W}_{D^{n}}\right)
\mapsto(\gamma_{+}\dot{-}\gamma_{-})\in M\otimes\mathcal{W}_{D}\right) \\
&  \otimes\mathrm{id}_{\mathcal{W}_{D_{m}}}\text{,}%
\end{align*}
the upper horizontal arrow represents
\[
\left(  \gamma_{+},\gamma_{-}\right)  \in\left(  M\otimes\mathcal{W}_{D^{n}%
}\right)  \underset{M\otimes\mathcal{W}_{D\left\{  n\right\}  _{n-1}}}{\times
}\left(  M\otimes\mathcal{W}_{D^{n}}\right)  \mapsto(\gamma_{+}\dot{-}%
\gamma_{-})\in M\otimes\mathcal{W}_{D}\text{,}%
\]
the left vertical arrow represents the composition of mappings
\begin{align*}
&  \left(  M\otimes\mathcal{W}_{D^{n}}\right)  \underset{M\otimes
\mathcal{W}_{D\left\{  n\right\}  _{n-1}}}{\times}\left(  M\otimes
\mathcal{W}_{D^{n}}\right) \\
&  \underline{\left(  \mathrm{id}_{M}\otimes\mathcal{W}_{\left(
d_{1},...,d_{n},e\right)  \in D^{n}\times D_{m}\mapsto\left(  d_{1}%
,...,ed_{i},...,d_{n}\right)  \in D^{n}}\right)  \times}\\
&  \underrightarrow{\left(  \mathrm{id}_{M}\otimes\mathcal{W}_{\left(
d_{1},...,d_{n},e\right)  \in D^{n}\times D_{m}\mapsto\left(  d_{1}%
,...,ed_{i},...,d_{n}\right)  \in D^{n}}\right)  }\\
&  \left(  M\otimes\mathcal{W}_{D^{n}\times D_{m}}\right)  \underset
{M\otimes\mathcal{W}_{D\left\{  n\right\}  _{n-1}\times D_{m}}}{\times}\left(
M\otimes\mathcal{W}_{D^{n}\times D_{m}}\right) \\
&  =\left(  \left(  M\otimes\mathcal{W}_{D^{n}}\right)  \otimes\mathcal{W}%
_{D_{m}}\right)  \underset{\left(  M\otimes\mathcal{W}_{D\left\{  n\right\}
_{n-1}}\right)  \otimes\mathcal{W}_{D_{m}}}{\times}\left(  \left(
M\otimes\mathcal{W}_{D^{n}}\right)  \otimes\mathcal{W}_{D_{m}}\right) \\
&  =\left(  \left(  M\otimes\mathcal{W}_{D^{n}}\right)  \underset
{M\otimes\mathcal{W}_{D\left\{  n\right\}  _{n-1}}}{\times}\left(
M\otimes\mathcal{W}_{D^{n}}\right)  \right)  \otimes\mathcal{W}_{D_{m}%
}\text{,}%
\end{align*}
and the right vertical arrow represents the composition of mappings
\[
M\otimes\mathcal{W}_{D}\underrightarrow{\mathrm{id}_{M}\otimes\mathcal{W}%
_{\left(  d,e\right)  \in D\times D_{m}\mapsto de\in D}}M\otimes
\mathcal{W}_{D\times D_{m}}=\left(  M\otimes\mathcal{W}_{D}\right)
\otimes\mathcal{W}_{D_{m}}\text{;}%
\]

\item In the latter diagram, the lower horizontal arrow represents
\begin{align*}
&  \left(  \left(  t,\gamma\right)  \in\left(  M\otimes\mathcal{W}_{D}\right)
\underset{M}{\times}\left(  M\otimes\mathcal{W}_{D^{n}}\right)  \mapsto
t\dot{+}\gamma\in M\otimes\mathcal{W}_{D^{n}}\right) \\
&  \otimes\mathrm{id}_{\mathcal{W}_{D_{m}}}\text{,}%
\end{align*}
the upper horizontal arrow represents
\[
\left(  t,\gamma\right)  \in\left(  M\otimes\mathcal{W}_{D}\right)
\underset{M}{\times}\left(  M\otimes\mathcal{W}_{D^{n}}\right)  \mapsto
t\dot{+}\gamma\in M\otimes\mathcal{W}_{D^{n}}\text{,}%
\]
the left vertical arrow represents the composition of mappings
\begin{align*}
&  \left(  M\otimes\mathcal{W}_{D}\right)  \underset{M}{\times}\left(
M\otimes\mathcal{W}_{D^{n}}\right) \\
&  \underrightarrow{\left(  \mathrm{id}_{M}\otimes\mathcal{W}_{\left(
d,e\right)  \in D\times D_{m}\mapsto ed\in D}\right)  \times\left(
\mathrm{id}_{M}\otimes\mathcal{W}_{\left(  d_{1},...,d_{n},e\right)  \in
D^{n}\times D_{m}\mapsto\left(  d_{1},...,ed_{i},...,d_{n}\right)  \in D^{n}%
}\right)  }\\
&  \left(  M\otimes\mathcal{W}_{D\times D_{m}}\right)  \underset
{M\otimes\mathcal{W}_{D\left\{  n\right\}  _{n-1}\times D_{m}}}{\times}\left(
M\otimes\mathcal{W}_{D^{n}\times D_{m}}\right) \\
&  =\left(  \left(  M\otimes\mathcal{W}_{D}\right)  \otimes\mathcal{W}_{D_{m}%
}\right)  \underset{\left(  M\otimes\mathcal{W}_{D\left\{  n\right\}  _{n-1}%
}\right)  \otimes\mathcal{W}_{D_{m}}}{\times}\left(  \left(  M\otimes
\mathcal{W}_{D^{n}}\right)  \otimes\mathcal{W}_{D_{m}}\right) \\
&  =\left(  \left(  M\otimes\mathcal{W}_{D}\right)  \underset{M\otimes
\mathcal{W}_{D\left\{  n\right\}  _{n-1}}}{\times}\left(  M\otimes
\mathcal{W}_{D^{n}}\right)  \right)  \otimes\mathcal{W}_{D_{m}}\text{,}%
\end{align*}
and the right vertical arrow represents the composition of mappings
\begin{align*}
&  M\otimes\mathcal{W}_{D^{n}}\underrightarrow{\mathrm{id}_{M}\otimes
\mathcal{W}_{\left(  d_{1},...,d_{n},e\right)  \in D^{n}\times D_{m}%
\mapsto\left(  d_{1},...,ed_{i},...,d_{n}\right)  \in D^{n}}}M\otimes
\mathcal{W}_{D^{n}\times D_{m}}\\
&  =\left(  M\otimes\mathcal{W}_{D^{n}}\right)  \otimes\mathcal{W}_{D_{m}%
}\text{.}%
\end{align*}

\end{enumerate}
\end{proposition}

\begin{proposition}
\label{t8.2.1.10}For any $\sigma\in\mathbf{S}_{n}$, any $\gamma_{+},\gamma
_{-},\gamma\in M\otimes\mathcal{W}_{D^{n}}\ $and any $t\in M\otimes
\mathcal{W}_{D}\ $with
\[
\left(  \mathrm{id}_{M}\otimes\mathcal{W}_{i_{D\left\{  n\right\}
_{n-1}\rightarrow D^{n}}}\right)  \left(  \gamma_{+}\right)  =\left(
\mathrm{id}_{M}\otimes\mathcal{W}_{i_{D\left\{  n\right\}  _{n-1}\rightarrow
D^{n}}}\right)  \left(  \gamma_{-}\right)
\]
and
\[
\left(  \mathrm{id}_{M}\otimes\mathcal{W}_{i_{1\rightarrow D}}\right)  \left(
t\right)  =\left(  \mathrm{id}_{M}\otimes\mathcal{W}_{i_{1\rightarrow D^{n}}%
}\right)  \left(  \gamma\right)  \text{,}%
\]
we have
\begin{align*}
\left(  \gamma_{+}\right)  ^{\sigma}\dot{-}\left(  \gamma_{-}\right)
^{\sigma}  &  =\gamma_{+}\dot{-}\gamma_{-}\\
\left(  t\dot{+}\gamma\right)  ^{\sigma}  &  =t\dot{+}\gamma^{\sigma}\text{.}%
\end{align*}

\end{proposition}

\begin{proposition}
\label{t8.2.1.11}For $\gamma_{+},\gamma_{-}\in M\otimes\mathcal{W}_{D^{n}}%
\ $with
\[
\left(  \mathrm{id}_{M}\otimes\mathcal{W}_{i_{D\left\{  n\right\}
_{n-1}\rightarrow D^{n}}}\right)  \left(  \gamma_{+}\right)  =\left(
\mathrm{id}_{M}\otimes\mathcal{W}_{i_{D\left\{  n\right\}  _{n-1}\rightarrow
D^{n}}}\right)  \left(  \gamma_{-}\right)  \text{,}%
\]
we have
\begin{align*}
&  \left(  \mathrm{id}_{M}\otimes\mathcal{W}_{\left(  d_{1},...,d_{n}\right)
\in D^{n}\mapsto d_{1}...d_{n}\in D}\right)  \left(  \gamma_{+}\dot{-}%
\gamma_{-}\right) \\
&  =(...(\gamma_{+}\underset{1}{-}\gamma_{-})\underset{2}{-}\mathbf{s}%
_{1}\circ\mathbf{d}_{1}(\gamma_{+}))\underset{3}{-}\mathbf{s}_{1}^{2}%
\circ\mathbf{d}_{1}^{2}(\gamma_{+}))...\underset{n}{-}\mathbf{s}_{1}%
^{n-1}\circ\mathbf{d}_{1}^{n-1}(\gamma_{+}))
\end{align*}

\end{proposition}

\subsubsection{\label{s8.2.2}Symmetric Forms}

\begin{definition}
\label{d8.2.2.1}A \textit{symmetric} $D^{n}$-\textit{form}$\mathit{\ }%
$\textit{at} $x\in E$ is a mapping $\omega_{x}:\left(  M\otimes\mathcal{W}%
_{D^{n}}\right)  _{\pi(x)}\rightarrow\left(  E\otimes\mathcal{W}_{D}\right)
_{x}^{\perp}\ $subject to the following conditions:

\begin{enumerate}
\item We have
\[
\omega_{x}(\alpha\underset{i}{\cdot}\gamma)=\alpha\omega_{x}(\gamma)\quad
\quad(1\leq i\leq n)
\]
for any $\gamma\in\left(  M\otimes\mathcal{W}_{D^{n}}\right)  _{\pi(x)}$ and
any $\alpha\in\mathbb{R}$.

\item The diagram
\[%
\begin{array}
[c]{ccc}%
\left(  M\otimes\mathcal{W}_{D^{n}}\right)  _{\pi(x)} & \rightarrow & \left(
M\otimes\mathcal{W}_{D^{n}}\right)  _{\pi(x)}\otimes\mathcal{W}_{D_{m}}\\
\omega_{x}\downarrow &  & \downarrow\omega_{x}\otimes\mathrm{id}%
_{\mathcal{W}_{D_{m}}}\\
\left(  E\otimes\mathcal{W}_{D}\right)  _{x} & \overrightarrow{\mathrm{id}%
_{E}\otimes\mathcal{W}_{\times_{D\times D_{m}\rightarrow D}}} & \left(
E\otimes\mathcal{W}_{D}\right)  _{x}\otimes\mathcal{W}_{D_{m}}%
\end{array}
\;(1\leq i\leq n)
\]
is commutative, where the upper horizontal arrow is
\[
\mathrm{id}_{M}\otimes\mathcal{W}_{\left(  d_{1},...,d_{n},e\right)  \in
D^{n}\times D_{m}\mapsto\left(  d_{1},...,d_{i-1},ed_{i},d_{i+1}%
,...d_{n}\right)  \in D^{n}}\text{.}%
\]

\item We have
\[
\omega_{x}\left(  \gamma^{\sigma}\right)  =\omega_{x}\left(  \gamma\right)
\]
for any $\gamma\in\left(  M\otimes\mathcal{W}_{D^{n}}\right)  _{\pi(x)}$ and
any $\sigma\in\mathbf{S}_{n}$.

\item We have
\[
\omega_{x}\left(  \left(  \mathrm{id}_{M}\otimes\mathcal{W}_{(d_{1}%
,...,d_{n})\in D^{n}\longmapsto(d_{1},...,d_{n-2},d_{n-1}d_{n})\in D^{n-1}%
}\right)  \left(  \gamma\right)  \right)  =0
\]
for any $\gamma\in\left(  M\otimes\mathcal{W}_{D^{n-1}}\right)  _{\pi(x)}$ .
\end{enumerate}
\end{definition}

\begin{notation}
We denote by $\mathbb{S}_{x}^{D^{n}}(\pi)\ $the totality of symmetric $D^{n}
$-\textit{forms}$\ $at $x\in E$. We denote by $\mathbb{S}^{D^{n}}(\pi
)\ $the\ set-theoretic union of $\mathbb{S}_{x}^{D^{n}}(\pi)$'s$\ $for all
$x\in E$. The canonical projection $\mathbb{S}^{D^{n}}(\pi)\rightarrow E$ is
obviously a vector bundle.
\end{notation}

\begin{proposition}
\label{t8.2.2.1}Let $\omega\in\mathbb{S}_{x}^{D^{n+1}}(\pi)$. Then we have
\[
\omega(\mathbf{s}_{i}(\gamma))=0\quad\quad(1\leq i\leq n+1)
\]
for any $\gamma\in\left(  M\otimes\mathcal{W}_{D^{n}}\right)  _{\pi(x)}$.
\end{proposition}

\begin{proof}
For any $\alpha\in\mathbb{R}$, we have
\[
\,\omega(\mathbf{s}_{i}(\gamma))=\omega(\alpha\underset{i}{\cdot}%
\mathbf{s}_{i}(\gamma))=\alpha\omega(\mathbf{s}_{i}(\gamma))
\]
Letting $\alpha=0$, we have the desired conclusion.
\end{proof}

\subsubsection{\label{s8.2.3}The Theorem}

The following proposition will be used in the proof of Proposition 3.6.

\begin{proposition}
\label{t8.2.3.1}Let $\nabla_{x}\in\mathbb{J}_{x}^{D^{n}}(\pi)$, $t\in\left(
M\otimes\mathcal{W}_{D}\right)  _{\pi\left(  x\right)  }$ and $\gamma
,\gamma_{+},\gamma_{-}\in\left(  M\otimes\mathcal{W}_{D^{n}}\right)
_{\pi\left(  x\right)  }$ with
\[
\left(  \mathrm{id}_{M}\otimes\mathcal{W}_{i_{D\left\{  n\right\}
_{n-1}\rightarrow D^{n}}}\right)  \left(  \gamma_{+}\right)  =\left(
\mathrm{id}_{M}\otimes\mathcal{W}_{i_{D\left\{  n\right\}  _{n-1}\rightarrow
D^{n}}}\right)  \left(  \gamma_{-}\right)  \text{.}%
\]
Then we have
\begin{align*}
\nabla_{x}(\gamma_{+})\dot{-}\nabla_{x}(\gamma_{-})  &  =\left(
\underline{\pi}_{n,1}(\nabla_{x})\right)  (\gamma_{+}\dot{-}\gamma_{-})\\
\left(  \pi_{n,1}(\nabla_{x})\right)  (t)\dot{+}\nabla_{x}(\gamma)  &
=\nabla_{x}(t\dot{+}\gamma)
\end{align*}

\end{proposition}

\begin{proof}
It is an easy exercise of affine geometry to show that the coveted two
formulas are equivalent. Here we deal only with the former in case of $n=2$,
leaving the general treatment safely to the reader. We have
\begin{align*}
&  \left(  \mathrm{id}_{E}\otimes\mathcal{W}_{(d_{1},d_{2})\in D^{2}\mapsto
d_{1}d_{2}\in D}\right)  (\nabla_{x}(\gamma_{+})\dot{-}\nabla_{x}(\gamma
_{-}))\\
&  =\left(  \nabla_{x}(\gamma_{+})\underset{1}{-}\nabla_{x}(\gamma
_{-})\right)  \underset{2}{-}\left(  \mathbf{s}_{1}\circ\mathbf{d}_{1}\right)
\left(  \nabla_{x}(\gamma_{+})\right) \\
&  \text{\lbrack By Proposition \ref{t8.2.1.11}]}\\
&  =\nabla_{x}((\gamma_{+}\underset{1}{-}\gamma_{-})\underset{2}{-}%
(\mathbf{s}_{1}\circ\mathbf{d}_{1})(\gamma_{+}))\\
&  =\nabla_{x}\left(  \left(  \mathrm{id}_{M}\otimes\mathcal{W}_{(d_{1}%
,d_{2})\in D^{2}\mapsto d_{1}d_{2}\in D}\right)  \left(  \gamma_{+}\dot
{-}\gamma_{-}\right)  \right) \\
&  \text{\lbrack By Proposition \ref{t8.2.1.11}]}\\
&  =\left(  \mathrm{id}_{E}\otimes\mathcal{W}_{(d_{1},d_{2})\in D^{2}\mapsto
d_{1}d_{2}\in D}\right)  \left(  \pi_{2,1}(\nabla_{x})(\gamma_{+}\dot{-}%
\gamma_{-})\right) \\
&  \text{\lbrack By Proposition \ref{t4.1.4}]}%
\end{align*}

\end{proof}

\begin{proposition}
\label{t8.2.3.2}Let $\nabla_{x}^{+},\nabla_{x}^{-}\in\mathbb{J}_{x}^{n+1}%
(\pi)$ with
\[
\pi_{n+1,n}(\nabla_{x}^{+})=\ \pi_{n+1,n}(\nabla_{x}^{-})\text{.}%
\]
Then the assignment $\gamma\in\left(  M\otimes\mathcal{W}_{D^{n+1}}\right)
_{\pi\left(  x\right)  }\longmapsto\nabla_{x}^{+}(\gamma)\dot{-}\nabla_{x}%
^{-}(\gamma)$ belongs to $\mathbb{S}_{x}^{D^{n+1}}(\pi)$.
\end{proposition}

\begin{proof}
\begin{enumerate}
\item Since
\begin{align*}
&  \left(  \pi\otimes\mathrm{id}_{\mathcal{W}_{D}}\right)  \left(  \nabla
_{x}^{+}(\gamma)\dot{-}\nabla_{x}^{-}(\gamma)\right) \\
&  =\left(  \pi\otimes\mathrm{id}_{\mathcal{W}_{D^{n+1}}}\right)  \left(
\nabla_{x}^{+}(\gamma)\right)  \dot{-}\left(  \pi\otimes\mathrm{id}%
_{\mathcal{W}_{D^{n+1}}}\right)  \left(  \nabla_{x}^{-}(\gamma)\right) \\
&  \text{[By Proposition \ref{t8.2.1.3}]}\\
&  =0\text{,}%
\end{align*}
$\nabla_{x}^{+}(\gamma)\dot{-}\nabla_{x}^{-}(\gamma)$ belongs in $\left(
E\otimes\mathcal{W}_{D}\right)  _{x}^{\perp}$.

\item For any $\alpha\in\mathbb{R}$ and any natural number $i$ with
$1\leq1\leq n+1$, we have
\begin{align*}
&  \nabla_{x}^{+}(\alpha\underset{\dot{i}}{\cdot}\gamma)\dot{-}\nabla_{x}%
^{-}(\alpha\underset{\dot{i}}{\cdot}\gamma)\\
&  =\alpha\underset{i}{\cdot}\nabla_{x}^{+}(\gamma)\dot{-}\alpha\underset
{i}{\cdot}\nabla_{x}^{-}(\gamma)\\
&  =\alpha(\nabla_{x}^{+}(\gamma)\dot{-}\nabla_{x}^{-}(\gamma))\text{,}%
\end{align*}
which implies that the assignment
\[
\gamma\in\left(  M\otimes\mathcal{W}_{D^{n+1}}\right)  _{\pi\left(  x\right)
}\longmapsto\nabla_{x}^{+}(\gamma)\dot{-}\nabla_{x}^{-}(\gamma)\in\left(
E\otimes\mathcal{W}_{D}\right)  _{x}^{\perp}%
\]
abides by the first condition in Definition \ref{d8.2.2.1}.

\item To see that the assignment abides by the second condition in Definition
\ref{d8.2.2.1}, it suffices to note that the diagram
\begin{align*}
&
\begin{array}
[c]{ccc}%
\left(  M\otimes\mathcal{W}_{D^{n+1}}\right)  _{\pi\left(  x\right)  } &
\rightarrow_{i} & \left(  M\otimes\mathcal{W}_{D^{n+1}}\right)  _{\pi\left(
x\right)  }\otimes\mathcal{W}_{D_{m}}\\
\downarrow &  & \downarrow\\%
\begin{array}
[c]{c}%
\left(  E\otimes\mathcal{W}_{D^{n+1}}\right)  _{x}\underset{E\otimes
\mathcal{W}_{D\left\{  n+1\right\}  _{n}}}{\times}\\
\left(  E\otimes\mathcal{W}_{D^{n+1}}\right)  _{x}%
\end{array}
& \rightarrow_{i} &
\begin{array}
[c]{c}%
\left(  \left(  E\otimes\mathcal{W}_{D^{n+1}}\right)  _{x}\underset
{E\otimes\mathcal{W}_{D\left\{  n+1\right\}  _{n}}}{\times}\left(
E\otimes\mathcal{W}_{D^{n+1}}\right)  _{x}\right) \\
\otimes\mathcal{W}_{D_{m}}%
\end{array}
\\
\downarrow &  & \downarrow\\
\left(  E\otimes\mathcal{W}_{D}\right)  _{x} & \rightarrow & \left(
E\otimes\mathcal{W}_{D}\right)  _{x}\otimes\mathcal{W}_{D_{m}}%
\end{array}
\\
&  \left(  1\leq i\leq n+1)\right)
\end{align*}
is commutative, where the upper horizontal arrow is
\[
\mathrm{id}_{M}\otimes\mathcal{W}_{\left(  d_{1},...,d_{n+1},e\right)  \in
D^{n+1}\times D_{m}\mapsto\left(  d_{1},...,d_{i-1},ed_{i},d_{i+1}%
,...d_{n+1}\right)  \in D^{n+1}}\text{,}%
\]
the middle horizontal arrow is the mapping
\begin{align*}
&  \left(  E\otimes\mathcal{W}_{D^{n+1}}\right)  \underset{E\otimes
\mathcal{W}_{D\left\{  n+1\right\}  _{n}}}{\times}\left(  E\otimes
\mathcal{W}_{D^{n+1}}\right) \\
&  \underline{\left(  \mathrm{id}_{E}\otimes\mathcal{W}_{\left(
d_{1},...,d_{n+1},e\right)  \in D^{n+1}\times D_{m}\mapsto\left(
d_{1},...,d_{i-1},ed_{i},d_{i+1},...d_{n+1}\right)  \in D^{n+1}}\right)
\times}\\
&  \underrightarrow{\left(  \mathrm{id}_{E}\otimes\mathcal{W}_{\left(
d_{1},...,d_{n+1},e\right)  \in D^{n+1}\times D_{m}\mapsto\left(
d_{1},...,d_{i-1},ed_{i},d_{i+1},...d_{n+1}\right)  \in D^{n+1}}\right)  }\\
&  \left(  E\otimes\mathcal{W}_{D^{n+1}\times D_{m}}\right)  \underset
{E\otimes\mathcal{W}_{D\left\{  n+1\right\}  _{n}}}{\times}\left(
E\otimes\mathcal{W}_{D^{n+1}\times D_{m}}\right) \\
&  =\left(  \left(  E\otimes\mathcal{W}_{D^{n+1}}\right)  \underset
{E\otimes\mathcal{W}_{D\left\{  n+1\right\}  _{n}}}{\times}E\otimes
\mathcal{W}_{D^{n+1}}\right)  \otimes\mathcal{W}_{D_{m}}\text{,}%
\end{align*}
the lower horizontal arrow is
\[
\mathrm{id}_{E}\otimes\mathcal{W}_{\times_{D\times D_{m}\rightarrow D}%
}\text{,}%
\]
the upper left vertical arrow is
\begin{align*}
\gamma &  \in\left(  M\otimes\mathcal{W}_{D^{n+1}}\right)  _{\pi\left(
x\right)  }\mapsto\\
\left(  \nabla_{x}^{+}(\gamma),\nabla_{x}^{-}(\gamma)\right)   &  \in\left(
E\otimes\mathcal{W}_{D^{n+1}}\right)  _{x}\underset{E\otimes\mathcal{W}%
_{D\left\{  n+1\right\}  _{n}}}{\times}\left(  E\otimes\mathcal{W}_{D^{n+1}%
}\right)  _{x}\text{,}%
\end{align*}
the lower left vertical arrow is
\[
\left(  \gamma^{+},\gamma^{-}\right)  \in\left(  E\otimes\mathcal{W}_{D^{n+1}%
}\right)  _{x}\underset{E\otimes\mathcal{W}_{D\left\{  n+1\right\}  _{n}}%
}{\times}\left(  E\otimes\mathcal{W}_{D^{n+1}}\right)  _{x}\mapsto\gamma
^{+}\dot{-}\gamma^{-}\in\left(  E\otimes\mathcal{W}_{D}\right)  _{x}\text{,}%
\]
the upper right vertical arrow is obtained from the upper left vertical arrow
by multiplication of $\otimes\mathrm{id}_{\mathcal{W}_{D_{m}}}$ from the
right, and the lower right vertical arrow is obtained from the lower left
vertical arrow by multiplication of $\otimes\mathrm{id}_{\mathcal{W}_{D_{m}}}$
from the right. The upper square is commutative by the third condition in
Definition \ref{d4.1}, while the lower square is commutaive by Proposition
\ref{t8.2.1.9}, so that the outer square is also commutative, which is no
other than the second condition in Definition \ref{d8.2.2.1}.

\item For any $\sigma\in\mathbf{S}_{n+1}$, we have
\begin{align*}
&  \nabla_{x}^{+}(\gamma^{\sigma})\dot{-}\nabla_{x}^{-}(\gamma^{\sigma})\\
&  =\left(  \nabla_{x}^{+}(\gamma)\right)  ^{\sigma}\dot{-}\left(  \nabla
_{x}^{-}(\gamma)\right)  ^{\sigma}\\
&  =\nabla_{x}^{+}(\gamma)\dot{-}\nabla_{x}^{-}(\gamma)\text{,}%
\end{align*}
which implies that the assignment abides by the third condition in Definition
\ref{d8.2.2.1}.

\item It remains to show that the assignment abides by the fourth condition in
Definition \ref{d8.2.2.1}, which follows directly from the second condition in
Definition \ref{d4.2} and the assumption that $\underline{\hat{\pi}}%
_{n+1,n}(\nabla_{x}^{+})=\ \underline{\hat{\pi}}_{n+1,n}(\nabla_{x}^{-})$.
\end{enumerate}
\end{proof}

\begin{proposition}
\label{t8.2.3.3}Let $\nabla_{x}\in\mathbb{J}_{x}^{D^{n+1}}(\pi)$ and
$\omega\in\mathbb{S}_{x}^{D^{n+1}}(\pi)$. Then the assignment $\gamma
\in\left(  M\otimes\mathcal{W}_{D^{n+1}}\right)  _{\pi\left(  x\right)
}\longmapsto\omega(\gamma)\dot{+}\nabla_{x}(\gamma)$ belongs to $\mathbb{J}%
_{x}^{D^{n+1}}(\pi)$.
\end{proposition}

\begin{proof}
\begin{enumerate}
\item Since
\begin{align*}
&  \left(  \pi\otimes\mathrm{id}_{\mathcal{W}_{D^{n+1}}}\right)  \left(
\omega(\gamma)\dot{+}\nabla_{x}(\gamma)\right) \\
&  =\left(  \pi\otimes\mathrm{id}_{\mathcal{W}_{D}}\right)  \left(
\omega(\gamma)\right)  \dot{+}\left(  \pi\otimes\mathrm{id}_{\mathcal{W}%
_{D^{n+1}}}\right)  \left(  \nabla_{x}(\gamma)\right) \\
&  \text{[By Proposition \ref{t8.2.1.6}]}\\
&  =\gamma\text{,}%
\end{align*}
the assignment
\[
\gamma\in\left(  M\otimes\mathcal{W}_{D^{n+1}}\right)  _{\pi\left(  x\right)
}\longmapsto\omega(\gamma)\dot{+}\nabla_{x}(\gamma)
\]
stands to the first condition in Definition \ref{d4.1}.

\item For any $\alpha\in\mathbb{R}$ and any natural number $i$ with $1\leq
i\leq n+1$, we have
\begin{align*}
&  \omega(\alpha\underset{i}{\cdot}\gamma)\dot{+}\nabla_{x}(\alpha\underset
{i}{\cdot}\gamma)\\
&  =\alpha\omega(\gamma)\dot{+}\alpha\underset{i}{\cdot}\nabla_{x}(\gamma)\\
&  =\alpha\underset{i}{\cdot}(\omega(\gamma)\dot{+}\nabla_{x}(\gamma))\text{,}%
\end{align*}
so that the assignment stands to the second condition in Definition \ref{d4.1}.

\item To see that the assignment acquiesces in the third condition in
Definition \ref{d4.1}, it suffices to note that the diagram
\begin{align*}
&
\begin{array}
[c]{ccc}%
\left(  M\otimes\mathcal{W}_{D^{n+1}}\right)  _{\pi\left(  x\right)  } &
\rightarrow_{i} & \left(  M\otimes\mathcal{W}_{D^{n+1}}\right)  _{\pi\left(
x\right)  }\otimes\mathcal{W}_{D_{m}}\\
\downarrow &  & \downarrow\\
\left(  E\otimes\mathcal{W}_{D}\right)  _{x}\underset{E}{\times}\left(
E\otimes\mathcal{W}_{D^{n+1}}\right)  _{x} & \rightarrow_{i} & \left(  \left(
E\otimes\mathcal{W}_{D}\right)  \underset{E}{\times}\left(  E\otimes
\mathcal{W}_{D^{n+1}}\right)  _{x}\right)  \otimes\mathcal{W}_{D_{m}}\\
\downarrow &  & \downarrow\\
\left(  E\otimes\mathcal{W}_{D^{n+1}}\right)  _{x} & \rightarrow_{i} & \left(
E\otimes\mathcal{W}_{D^{n+1}}\right)  _{x}\otimes\mathcal{W}_{D_{m}}%
\end{array}
\\
&  \left(  1\leq i\leq n+1)\right)
\end{align*}
is commutative, where the upper horizontal arrow is
\[
\mathrm{id}_{M}\otimes\mathcal{W}_{\left(  \underset{i}{\cdot}\right)
_{D^{n+1}\times D_{m}}}\text{,}%
\]
the middle horizontal arrow is the composition of mappings
\begin{align*}
&  \left(  E\otimes\mathcal{W}_{D}\right)  \underset{M}{\times}\left(
E\otimes\mathcal{W}_{D^{n+1}}\right) \\
&  \underrightarrow{\left(  \mathrm{id}_{M}\otimes\mathcal{W}_{\times_{D\times
D_{m}\rightarrow D}}\right)  \times\left(  \mathrm{id}_{M}\otimes
\mathcal{W}_{\left(  \underset{i}{\cdot}\right)  _{D^{n+1}\times D_{m}}%
}\right)  }\\
&  \left(  E\otimes\mathcal{W}_{D\times D_{m}}\right)  \underset
{E\otimes\mathcal{W}_{D_{m}}}{\times}\left(  E\otimes\mathcal{W}%
_{D^{n+1}\times D_{m}}\right) \\
&  =\left(  \left(  E\otimes\mathcal{W}_{D}\right)  \underset{M}{\times
}\left(  E\otimes\mathcal{W}_{D^{n+1}}\right)  \right)  \otimes\mathcal{W}%
_{D_{m}}\text{,}%
\end{align*}
the lower horizontal arrow is
\[
\mathrm{id}_{E}\otimes\mathcal{W}_{\left(  \underset{i}{\cdot}\right)
_{D^{n+1}\times D_{m}}}\text{,}%
\]
the upper left vertical arrow is
\[
\gamma\in M\otimes\mathcal{W}_{D^{n+1}}\mapsto\left(  \omega_{x}%
(\gamma),\nabla_{x}(\gamma)\right)  \in\left(  E\otimes\mathcal{W}_{D}\right)
\underset{E}{\times}\left(  E\otimes\mathcal{W}_{D^{n+1}}\right)  \text{,}%
\]
he lower left vertical arrow is
\[
\left(  t,\gamma\right)  \in\left(  E\otimes\mathcal{W}_{D}\right)
\underset{E}{\times}\left(  E\otimes\mathcal{W}_{D^{n+1}}\right)  \mapsto
t\dot{+}\gamma\in E\otimes\mathcal{W}_{D^{n+1}}\text{,}%
\]
the upper right vertical arrow is the upper left vertical arrow multiplied by
$\otimes\mathrm{id}_{\mathcal{W}_{D_{m}}}$ from the right, and the lower right
vertical arrow is the lower left vertical arrow multiplied by $\otimes
\mathrm{id}_{\mathcal{W}_{D_{m}}}$ from the right. The upper square is
commutative by the third condition in Definition \ref{d4.1} and the second
condition in Definition \ref{d8.2.2.1}, while the lower square is commutaive
by Proposition \ref{t8.2.1.9}, so that the outer square is also commutative,
which is no other than the third condition in Definition \ref{d4.1}.

\item For any $\sigma\in\mathbf{S}_{n+1}$, we have
\begin{align*}
&  \,\omega(\gamma^{\sigma})\dot{+}\nabla_{x}(\gamma^{\sigma})\\
&  =\omega(\gamma)\dot{+}(\nabla_{x}(\gamma))^{\sigma}\\
&  =(\omega(\gamma)\dot{+}\nabla_{x}(\gamma))^{\sigma}\text{,}%
\end{align*}
so that the assignment stands to the fourth condition in Definition \ref{d4.1}.

\item That the assignment $\gamma\in\left(  M\otimes\mathcal{W}_{D^{n+1}%
}\right)  _{\pi\left(  x\right)  }\longmapsto\omega(\gamma)\dot{+}\nabla
_{x}(\gamma)$ stands to the first condition of Definition \ref{d4.2} follows
from the simple fact that the image of the assignment under $\underline
{\hat{\pi}}_{n+1,n}$ coincides with $\underline{\hat{\pi}}_{n+1,n}(\nabla
_{x})$, which is consequent upon Proposition \ref{t8.2.2.1}.

\item It remains to show that the assignment abides by the second condition in
Definition \ref{d4.2}, which follows directly from fourth condition in
Definition \ref{d8.2.2.1} and the second condition in Definition \ref{d4.2}.
\end{enumerate}
\end{proof}

Now we are in a position to give a definition.

\begin{definition}
\begin{enumerate}
\item For any $\nabla_{x}^{+},\nabla_{x}^{-}\in\mathbb{J}^{n+1}(\pi)$ with
\[
\pi_{n+1,n}(\nabla_{x}^{+})=\pi_{n+1,n}(\nabla_{x}^{-})\text{,}%
\]
we define $\nabla_{x}^{+}\dot{-}\nabla_{x}^{-}\in\mathbb{S}_{x}^{D_{n+1}}%
(\pi)$ to be
\[
(\nabla_{x}^{+}\dot{-}\nabla_{x}^{-})(\gamma)=\nabla_{x}^{+}(\gamma)\dot
{-}\nabla_{x}^{-}(\gamma)
\]
for any $\gamma\in\left(  M\otimes\mathcal{W}_{D^{n+1}}\right)  _{\pi\left(
x\right)  }$.

\item For any $\omega\in\mathbb{S}_{x}^{D^{n+1}}(\pi)\ $and any $\nabla_{x}%
\in\mathbb{J}^{n+1}(\pi)$, we define $\omega\dot{+}\nabla_{x}\in\mathbb{J}%
_{x}^{n+1}(\pi)$ to be
\[
(\omega\dot{+}\nabla_{x})(\gamma)=\omega(\gamma)\dot{+}\nabla_{x}(\gamma)
\]
for any $\gamma\in\left(  M\otimes\mathcal{W}_{D^{n+1}}\right)  _{\pi\left(
x\right)  }$.
\end{enumerate}
\end{definition}

With these two operations depicted in the above definition, it is easy to see that

\begin{theorem}
\label{t8.2.3.4}(cf. Theorem 6.2.9 of \cite{sa}). The bundle $\pi
_{n+1,n}:\mathbb{J}^{n+1}(\pi)\rightarrow\mathbb{J}^{n}(\pi)$ is an affine
bundle over the vector bundle $\mathbb{J}^{n}(\pi)\underset{E}{\times
}\mathbb{S}^{D^{n+1}}(\pi)\rightarrow\mathbb{J}^{n}(\pi)$.
\end{theorem}

\begin{proof}
This follows simply from Theorem \ref{t8.2.1.7}.
\end{proof}

\subsection{\label{s8.3}The Theorem in the Third Approach}

\subsubsection{\label{s8.3.1}Affine Bundles}

Now we turn to another kind of affine bundles, for which we can proceed in the
same way as in Subsubsection \ref{s8.2.1}.

\begin{lemma}
\label{t8.3.1.1}The diagram
\[%
\begin{array}
[c]{ccc}%
D_{n} & \underrightarrow{i_{D_{n}\rightarrow D_{n+1}}} & D_{n+1}\\
i_{D_{n}\rightarrow D_{n+1}}\downarrow &  & \downarrow\Psi_{D_{n+1}}\\
D_{n+1} & \underrightarrow{\Phi_{D_{n+1}}} & D_{n+1}\oplus D
\end{array}
\]
is a quasi-colimit diagram, where $i:D_{n}\rightarrow D_{n+1}$ is the
canonical injection, $\Phi_{D_{n+1}}(d)=(d,0)$ and $\Psi_{D_{n+1}%
}(d)=(d,d^{n+1})$.
\end{lemma}

This implies directly that

\begin{proposition}
\label{t8.3.1.2}Given $\gamma_{+},\gamma_{-}\in$ $M\otimes\mathcal{W}%
_{D_{n+1}}$ with
\[
\left(  \mathrm{id}_{M}\otimes\mathcal{W}_{i_{D_{n}\rightarrow D_{n+1}}%
}\right)  \left(  \gamma_{+}\right)  =\left(  \mathrm{id}_{M}\otimes
\mathcal{W}_{i_{D_{n}\rightarrow D_{n+1}}}\right)  \left(  \gamma_{-}\right)
\text{,}%
\]
there exists unique $\gamma\in M\otimes\mathcal{W}_{D_{n+1}\oplus D}$ with
\begin{align*}
\left(  \mathrm{id}_{M}\otimes\mathcal{W}_{\Psi_{D_{n+1}}}\right)  \left(
\gamma\right)   &  =\gamma_{+}\text{ and}\\
\left(  \mathrm{id}_{M}\otimes\mathcal{W}_{\Phi_{D_{n+1}}}\right)  \left(
\gamma\right)   &  =\gamma_{-}%
\end{align*}

\end{proposition}

\begin{notation}
Under the same notation as in the above proposition, we denote
\[
\left(  \mathrm{id}_{M}\otimes\mathcal{W}_{\Xi_{D_{n+1}}}\right)  \left(
\gamma\right)
\]
by $\gamma_{+}\dot{-}\gamma_{-}$, where $\Xi_{D_{n+1}}:D\rightarrow
D^{n+1}\oplus D$ is the mapping
\[
d\in D\mapsto\left(  0,...,0,d\right)  \in D^{n+1}\oplus D
\]

\end{notation}

From the very definition of $\dot{-}$, we have

\begin{proposition}
\label{t8.3.1.3}Let $\varphi$ be a mapping of $M$\ into $M^{\prime}$. Given
$\gamma_{+},\gamma_{-}\in M\otimes\mathcal{W}_{D_{n+1}}$ with
\[
\left(  \mathrm{id}_{M}\otimes\mathcal{W}_{i_{D_{n}\rightarrow D_{n+1}}%
}\right)  \left(  \gamma_{+}\right)  =\left(  \mathrm{id}_{M}\otimes
\mathcal{W}_{i_{D_{n}\rightarrow D_{n+1}}}\right)  \left(  \gamma_{-}\right)
\text{,}%
\]
we have
\begin{align*}
&  \left(  \mathrm{id}_{M^{\prime}}\otimes\mathcal{W}_{i_{D_{n}\rightarrow
D_{n+1}}}\right)  \left(  \left(  \varphi\otimes\mathrm{id}_{\mathcal{W}%
_{D_{n+1}}}\right)  \left(  \gamma_{+}\right)  \right) \\
&  =\left(  \mathrm{id}_{M^{\prime}}\otimes\mathcal{W}_{i_{D_{n}\rightarrow
D_{n+1}}}\right)  \left(  \left(  \varphi\otimes\mathrm{id}_{\mathcal{W}%
_{D_{n+1}}}\right)  \left(  \gamma_{-}\right)  \right)
\end{align*}
and
\begin{align*}
&  \left(  \varphi\otimes\mathrm{id}_{\mathcal{W}_{D}}\right)  \left(
\gamma_{+}\dot{-}\gamma_{-}\right) \\
&  =\left(  \varphi\otimes\mathrm{id}_{\mathcal{W}_{D_{n+1}}}\right)  \left(
\gamma_{+}\right)  \dot{-}\left(  \varphi\otimes\mathrm{id}_{\mathcal{W}%
_{D_{n+1}}}\right)  \left(  \gamma_{-}\right)
\end{align*}

\end{proposition}

It is easy to see that

\begin{proposition}
\label{t8.3.1.4}

\begin{enumerate}
\item We have
\[
\alpha\gamma_{+}\dot{-}\alpha\gamma_{-}=\alpha^{n+1}(\gamma_{+}\dot{-}%
\gamma_{-})\text{ \ \ }(1\leq i\leq n+1)
\]
for any $\alpha\in\mathbb{R}$ and any $\gamma_{\pm}\in$ $M\otimes
\mathcal{W}_{D_{n+1}}$ with
\[
\left(  \mathrm{id}_{M}\otimes\mathcal{W}_{i_{D_{n}\rightarrow D_{n+1}}%
}\right)  \left(  \gamma_{+}\right)  =\left(  \mathrm{id}_{M}\otimes
\mathcal{W}_{i_{D_{n}\rightarrow D_{n+1}}}\right)  \left(  \gamma_{-}\right)
\text{.}%
\]

\item The diagram
\[%
\begin{array}
[c]{ccc}%
\left(  M\otimes\mathcal{W}_{D_{n+1}}\right)  \underset{M\otimes
\mathcal{W}_{D_{n}}}{\times}\left(  M\otimes\mathcal{W}_{D_{n+1}}\right)  &
\rightarrow &
\begin{array}
[c]{c}%
\left(  \left(  M\otimes\mathcal{W}_{D_{n+1}}\right)  \underset{M\otimes
\mathcal{W}_{D_{n}}}{\times}\left(  M\otimes\mathcal{W}_{D_{n+1}}\right)
\right) \\
\otimes\mathcal{W}_{D_{m}}%
\end{array}
\\
\downarrow &  & \downarrow\\
M\otimes\mathcal{W}_{D} & \rightarrow & M\otimes\mathcal{W}_{D\times D_{m}}%
\end{array}
\]
commutes, where the upper horizontal arrow is the composition of mappings
\begin{align*}
&  \left(  M\otimes\mathcal{W}_{D_{n+1}}\right)  \underset{M\otimes
\mathcal{W}_{D_{n}}}{\times}\left(  M\otimes\mathcal{W}_{D_{n+1}}\right) \\
&  \underrightarrow{\left(  \mathrm{id}_{M}\otimes\mathcal{W}_{\times
_{D_{n+1}\times D_{m}\rightarrow D_{n+1}}}\right)  \times\left(
\mathrm{id}_{M}\otimes\mathcal{W}_{\times_{D_{n+1}\times D_{m}\rightarrow
D_{n+1}}}\right)  }\\
&  \left(  M\otimes\mathcal{W}_{D_{n+1\times D_{m}}}\right)  \underset
{M\otimes\mathcal{W}_{D_{n\times D_{m}}}}{\times}\left(  M\otimes
\mathcal{W}_{D_{n+1\times D_{m}}}\right) \\
&  =\left(  \left(  M\otimes\mathcal{W}_{D_{n+1}}\right)  \underset
{M\otimes\mathcal{W}_{D_{n}}}{\times}\left(  M\otimes\mathcal{W}_{D_{n+1}%
}\right)  \right)  \otimes\mathcal{W}_{D_{m}}\text{,}%
\end{align*}
the lower horizontal arrow is
\[
\mathrm{id}_{M\otimes\mathcal{W}_{D}}\otimes\mathcal{W}_{d\in D_{m}\mapsto
d^{n}\in D_{m}}\text{,}%
\]
the left vertical arrow is
\[
\left(  \gamma_{+},\gamma_{-}\right)  \in\left(  M\otimes\mathcal{W}_{D_{n+1}%
}\right)  \underset{M\otimes\mathcal{W}_{D_{n}}}{\times}\left(  M\otimes
\mathcal{W}_{D_{n+1}}\right)  \mapsto\gamma_{+}\dot{-}\gamma_{-}\in
M\otimes\mathcal{W}_{D}\text{,}%
\]
and the right vertical arrow is
\begin{align*}
&  \left(  \left(  \gamma_{+},\gamma_{-}\right)  \in\left(  M\otimes
\mathcal{W}_{D_{n+1}}\right)  \underset{M\otimes\mathcal{W}_{D_{n}}}{\times
}\left(  M\otimes\mathcal{W}_{D_{n+1}}\right)  \mapsto\gamma_{+}\dot{-}%
\gamma_{-}\in M\otimes\mathcal{W}_{D}\right) \\
&  \otimes\mathrm{id}_{\mathcal{W}_{D_{m}}\text{.}}%
\end{align*}

\end{enumerate}
\end{proposition}

\begin{lemma}
\label{t8.3.1.5}The diagram
\[%
\begin{array}
[c]{ccc}%
1 & \underrightarrow{i_{1\rightarrow D}} & D\\
i_{1\rightarrow D_{n+1}}\downarrow &  & \downarrow\Xi_{D_{n+1}}\\
D_{n+1} & \underrightarrow{\Phi_{D_{n+1}}} & D_{n+1}\oplus D
\end{array}
\]
is a quasi-colimit diagram, where $i_{1\rightarrow D_{n+1}}\ $is the canonical injection.
\end{lemma}

This implies at once that

\begin{proposition}
\label{t8.3.1.6}Given $t\in M\otimes\mathcal{W}_{D}\ $and $\gamma\in
M\otimes\mathcal{W}_{D_{n+1}}\ $with
\[
\left(  \mathrm{id}_{M}\otimes\mathcal{W}_{i_{1\rightarrow D_{n+1}}}\right)
\left(  \gamma\right)  =\left(  \mathrm{id}_{M}\otimes\mathcal{W}%
_{i_{1\rightarrow D}}\right)  \left(  t\right)  \text{,}%
\]
there exists a unique function $\widetilde{\gamma}:D_{n+1}\oplus D\rightarrow
M\ $with
\[
\left(  \mathrm{id}_{M}\otimes\mathcal{W}_{\Phi_{D_{n+1}}}\right)  \left(
\widetilde{\gamma}\right)  =\gamma
\]
and
\[
\left(  \mathrm{id}_{M}\otimes\mathcal{W}_{\Xi_{D_{n+1}}}\right)  \left(
\widetilde{\gamma}\right)  =t\text{.}%
\]

\end{proposition}

\begin{notation}
Under the same notation as in the above proposition, we denote
\[
\left(  \mathrm{id}_{M}\otimes\mathcal{W}_{\Psi_{D_{n+1}}}\right)  \left(
\widetilde{\gamma}\right)
\]
by $t\dot{+}\gamma$.
\end{notation}

From the very definition of $\dot{+}\ $we have

\begin{proposition}
\label{t8.3.1.7}Let $\varphi$ be a mapping of $M$\ into $M^{\prime}$. Given
$t\in M\otimes\mathcal{W}_{D}\ $and $\gamma\in M\otimes\mathcal{W}_{D_{n+1}%
}\ $with
\[
\left(  \mathrm{id}_{M}\otimes\mathcal{W}_{i_{1\rightarrow D_{n+1}}}\right)
\left(  \gamma\right)  =\left(  \mathrm{id}_{M}\otimes\mathcal{W}%
_{i_{1\rightarrow D}}\right)  \left(  t\right)  \text{,}%
\]
we have
\[
\left(  \mathrm{id}_{M^{\prime}}\otimes\mathcal{W}_{i_{1\rightarrow D}%
}\right)  \left(  \left(  \varphi\otimes\mathrm{id}_{\mathcal{W}_{D}}\right)
\left(  t\right)  \right)  =\left(  \mathrm{id}_{M^{\prime}}\otimes
\mathcal{W}_{i_{1\rightarrow D_{n+1}}}\right)  \left(  \left(  \varphi
\otimes\mathrm{id}_{\mathcal{W}_{D_{n+1}}}\right)  \left(  \gamma\right)
\right)
\]
and
\[
\left(  \varphi\otimes\mathrm{id}_{\mathcal{W}_{D_{n+1}}}\right)  \left(
t\dot{+}\gamma\right)  =\left(  \varphi\otimes\mathrm{id}_{\mathcal{W}_{D}%
}\right)  \left(  t\right)  \dot{+}\left(  \varphi\otimes\mathrm{id}%
_{\mathcal{W}_{D_{n+1}}}\right)  \left(  \gamma\right)
\]

\end{proposition}

Now we have the following affine bundle theorem.

\begin{theorem}
\label{t8.3.1.8}The canonical projection $\mathrm{id}_{M}\otimes
\mathcal{W}_{i_{D_{n}\rightarrow D_{n+1}}}:M\otimes\mathcal{W}_{D_{n+1}%
}\mathcal{\rightarrow}M\otimes\mathcal{W}_{D_{n}}\ $is an affine bundle over
the vector bundle $\left(  M\otimes\mathcal{W}_{D}\right)  \underset{M}%
{\times}\left(  M\otimes\mathcal{W}_{D_{n}}\right)  \rightarrow M\otimes
\mathcal{W}_{D_{n}}$.
\end{theorem}

\subsubsection{\label{s8.3.2}Symmetric Forms}

\begin{definition}
\label{d8.3.2.1}A \textit{symmetric} $D_{n}$-\textit{form}$\mathit{\ }%
$\textit{at} $x\in E$ is a mapping $\omega_{x}:\left(  M\otimes\mathcal{W}%
_{D_{n}}\right)  _{\pi(x)}\rightarrow\left(  E\otimes\mathcal{W}_{D}\right)
_{x}^{\perp}\ $subject to the following conditions:

\begin{enumerate}
\item For any $\gamma\in\left(  M\otimes\mathcal{W}_{D_{n}}\right)  _{\pi(x)}$
and any $\alpha\in\mathbb{R}$,$\ $we have
\[
\omega_{x}(\alpha\gamma)=\alpha^{n}\omega_{x}(\gamma)
\]

\item The diagram
\[%
\begin{array}
[c]{ccc}%
\left(  M\otimes\mathcal{W}_{D_{n}}\right)  _{\pi\left(  x\right)  } &
\underrightarrow{\mathrm{id}_{M}\otimes\mathcal{W}_{\times_{D_{n}\times
D_{m}\rightarrow D_{n}}}} & \left(  M\otimes\mathcal{W}_{D_{n}}\right)
_{\pi\left(  x\right)  }\otimes\mathcal{W}_{D_{m}}\\
\omega_{x}\downarrow &  & \downarrow\omega_{x}\otimes\mathrm{id}%
_{\mathcal{W}_{D_{m}}}\\
\left(  E\otimes\mathcal{W}_{D}\right)  _{x} & \overrightarrow{\mathrm{id}%
_{E}\otimes\mathcal{W}_{\left(  d,e\right)  \in D\times D_{m}\mapsto de^{n}\in
D}} & \left(  E\otimes\mathcal{W}_{D}\right)  _{x}\otimes\mathcal{W}_{D_{m}}%
\end{array}
\]
is commutative.

\item For any simple polynomial $\rho$ of $d\in D_{n}$ and any $\gamma
\in\left(  M\otimes\mathcal{W}_{D_{l}}\right)  _{\pi(x)}$ with $\mathrm{\dim
}_{n}\rho=l<n$, we have
\[
\omega(\left(  \mathrm{id}_{M}\otimes\mathcal{W}_{\rho}\right)  \left(
\gamma\right)  )=0
\]

\end{enumerate}
\end{definition}

\begin{notation}
We denote by $\mathbb{S}_{x}^{D_{n}}(\pi)\ $the totality of symmetric $D_{n}
$-\textit{forms}$\ $at $x\in E$. We denote by $\mathbb{S}^{D_{n}}(\pi
)\ $the\ set-theoretic union of $\mathbb{S}_{x}^{D_{n}}(\pi)$'s$\ $for all
$x\in E$. The canonical projection $\mathbb{S}^{^{D_{n}}}(\pi)\rightarrow E$
is obviously a vector bundle.
\end{notation}

\subsubsection{\label{s8.3.3}The Theorem}

Now we turn to a variant of Theorem \ref{t8.2.3.4}, for which we can proceed
as in \ref{s8.2.3}, so that proofs of the following results are omitted or
merely indicated.

\begin{proposition}
\label{t8.3.3.1}Let $\nabla^{+}$, $\nabla^{-}\in\mathbb{J}_{x}^{D_{n+1}}(\pi)$
with $\pi_{n+1,n}(\nabla^{+})=\ \pi_{n+1,n}(\nabla^{-})$. Then the assignment
$\gamma\in\left(  M\otimes\mathcal{W}_{D_{n+1}}\right)  _{\pi\left(  x\right)
}\longmapsto\nabla^{+}(\gamma)\dot{-}\nabla^{-}(\gamma)\in\left(
E\otimes\mathcal{W}_{D}\right)  _{x}$ belongs to $\mathbb{S}_{x}^{D_{n+1}}%
(\pi)$.
\end{proposition}

\begin{proposition}
\label{t8.3.3.2}Let $\nabla\in\mathbb{J}_{x}^{D_{n+1}}(\pi)$ and $\omega
\in\mathbb{S}_{x}^{D_{n+1}}(\pi)$. Then the assignment $\gamma\in\left(
M\otimes\mathcal{W}_{D_{n+1}}\right)  _{\pi\left(  x\right)  }\longmapsto
\omega(\gamma)\dot{+}\nabla(\gamma)\in\left(  E\otimes\mathcal{W}_{D_{n+1}%
}\right)  _{x}$ belongs to $\mathbb{J}_{x}^{D_{n+1}}(\pi)$.
\end{proposition}

\begin{notation}
\begin{enumerate}
\item For any $\nabla^{+},\nabla^{-}\in\mathbb{J}^{D_{n+1}}(\pi)$ with
$\hat{\pi}_{n+1,n}(\nabla^{+})=\hat{\pi}_{n+1,n}(\nabla^{-})$, we define
$\nabla^{+}\dot{-}\nabla^{-}\in\mathbb{S}^{D_{n+1}}(\pi)$ to be
\[
(\nabla^{+}\dot{-}\nabla^{-})(\gamma)=\nabla^{+}(\gamma)\dot{-}\nabla
^{-}(\gamma)
\]
for any $\gamma\in\left(  M\otimes\mathcal{W}_{D_{n+1}}\right)  _{\pi\left(
x\right)  }$.

\item For any $\nabla\in\mathbb{J}_{x}^{D_{n+1}}(\pi)\ $and any $\omega
\in\mathbb{S}_{x}^{D_{n+1}}(\pi)$ we define $\omega\dot{+}\nabla\in
\mathbb{J}_{x}^{D_{n+1}}(\pi)$ to be
\[
(\omega\dot{+}\nabla)(\gamma)=\omega(\gamma)\dot{+}\nabla(\gamma)
\]
for any $\gamma\in\left(  M\otimes\mathcal{W}_{D_{n+1}}\right)  _{\pi\left(
x\right)  }$.
\end{enumerate}
\end{notation}

With these two operations, we have

\begin{theorem}
\label{t8.3.3.3}The bundle $\pi_{n+1,n}:\mathbb{J}^{D_{n+1}}(\pi
)\rightarrow\mathbb{J}^{D_{n}}(\pi)$ is an affine bundle over the vector
bundle $\mathbb{S}^{D_{n+1}}(\pi)\underset{E}{\times}\mathbb{J}^{D_{n}}%
(\pi)\rightarrow\mathbb{J}^{D_{n}}(\pi)$.
\end{theorem}

\begin{proof}
This follows simply from Theorem \ref{t8.3.1.8}.
\end{proof}

\subsection{\label{s8.4}The Comparison between the Second and Third
Approaches}

Now we are in a position to investigate the relationship between the affine
bundles discussed in Theorems \ref{t8.2.3.4} and \ref{t8.3.3.3}. Let us begin with

\begin{lemma}
\label{t8.4.1}Let $\gamma^{\pm}\in\left(  E\otimes\mathcal{W}_{D_{n+1}%
}\right)  _{x}$ with
\[
\left(  \mathrm{id}_{E}\otimes\mathcal{W}_{i_{D_{n}\rightarrow D_{n+1}}%
}\right)  \left(  \gamma_{+}\right)  =\left(  \mathrm{id}_{E}\otimes
\mathcal{W}_{i_{D_{n}\rightarrow D_{n+1}}}\right)  \left(  \gamma_{-}\right)
\text{.}%
\]
Then
\begin{align*}
&  \left(  \mathrm{id}_{E}\otimes\mathcal{W}_{i_{D\left\{  n+1\right\}
_{n}\rightarrow D^{n+1}}}\right)  \left(  \left(  \mathrm{id}_{E}%
\otimes\mathcal{W}_{+_{D^{n+1}\rightarrow D_{n+1}}}\right)  \left(  \gamma
_{+}\right)  \right) \\
&  =\left(  \mathrm{id}_{E}\otimes\mathcal{W}_{i_{D\left\{  n+1\right\}
_{n}\rightarrow D^{n+1}}}\right)  \left(  \left(  \mathrm{id}_{E}%
\otimes\mathcal{W}_{+_{D^{n+1}\rightarrow D_{n+1}}}\right)  \left(  \gamma
_{-}\right)  \right)
\end{align*}
obtains, and we have
\begin{align*}
&  \gamma^{+}\dot{-}\gamma^{-}\\
&  =\left(  \mathrm{id}_{E}\otimes\mathcal{W}_{+_{D^{n+1}\rightarrow D_{n+1}}%
}\right)  \left(  \gamma^{+}\right)  \dot{-}\left(  \mathrm{id}_{E}%
\otimes\mathcal{W}_{+_{D^{n+1}\rightarrow D_{n+1}}}\right)  \left(  \gamma
^{-}\right)  \text{.}%
\end{align*}

\end{lemma}

\begin{proof}
Since the diagram
\begin{equation}%
\begin{array}
[c]{ccc}%
D\left\{  n+1\right\}  _{n} & \underrightarrow{i_{D\left\{  n+1\right\}
_{n}\rightarrow D^{n+1}}} & D^{n+1}\\%
\begin{array}
[c]{cc}%
+_{D\left\{  n+1\right\}  _{n}\rightarrow D_{n}} & \downarrow
\end{array}
&  &
\begin{array}
[c]{cc}%
\downarrow & +_{D^{n+1}\rightarrow D_{n+1}}%
\end{array}
\\
D_{n} & \overrightarrow{i_{D_{n}\rightarrow D_{n+1}}} & D_{n+1}%
\end{array}
\label{8.4.1.1}%
\end{equation}
is commutative, we have
\begin{align*}
&  \left(  \mathrm{id}_{E}\otimes\mathcal{W}_{i_{D\left\{  n+1\right\}
_{n}\rightarrow D^{n+1}}}\right)  \left(  \left(  \mathrm{id}_{E}%
\otimes\mathcal{W}_{+_{D^{n+1}\rightarrow D_{n+1}}}\right)  \left(  \gamma
^{+}\right)  \right) \\
&  =\left(  \mathrm{id}_{E}\otimes\mathcal{W}_{+_{D\left\{  n+1\right\}
_{n}\rightarrow D_{n}}}\right)  \left(  \left(  \mathrm{id}_{E}\otimes
\mathcal{W}_{i_{D_{n}\rightarrow D_{n+1}}}\right)  \left(  \gamma^{+}\right)
\right) \\
&  =\left(  \mathrm{id}_{E}\otimes\mathcal{W}_{+_{D\left\{  n+1\right\}
_{n}\rightarrow D_{n}}}\right)  \left(  \left(  \mathrm{id}_{E}\otimes
\mathcal{W}_{i_{D_{n}\rightarrow D_{n+1}}}\right)  \left(  \gamma^{-}\right)
\right) \\
&  =\left(  \mathrm{id}_{E}\otimes\mathcal{W}_{i_{D\left\{  n+1\right\}
_{n}\rightarrow D^{n+1}}}\right)  \left(  \left(  \mathrm{id}_{E}%
\otimes\mathcal{W}_{+_{D^{n+1}\rightarrow D_{n+1}}}\right)  \left(  \gamma
^{-}\right)  \right)  \text{,}%
\end{align*}
which establishes the coveted first statement. The second statement follows
simply from a commutative cubical diagram, which is depicted here separately
as the upper square (\ref{8.4.1.1}), the lower square and the rounding four
side squares:
\begin{equation}%
\begin{array}
[c]{ccc}%
D^{n+1} & \underrightarrow{\Phi_{D^{n+1}}} & D^{n+1}\oplus D\\%
\begin{array}
[c]{cc}%
+_{D^{n+1}\rightarrow D_{n+1}} & \downarrow
\end{array}
&  &
\begin{array}
[c]{cc}%
\downarrow & +_{D^{n+1}\rightarrow D_{n+1}}\oplus\mathrm{id}_{D}%
\end{array}
\\
D_{n+1} & \overrightarrow{\Phi_{D_{n+1}}} & D_{n+1}\oplus D
\end{array}
\label{8.4.1.2}%
\end{equation}
\begin{equation}%
\begin{array}
[c]{ccccc}%
D\left\{  n+1\right\}  _{n} & \underrightarrow{i_{D\left\{  n+1\right\}
_{n}\rightarrow D^{n+1}}} & D^{n+1} & \underrightarrow{+_{D^{n+1}\rightarrow
D_{n+1}}} & D_{n+1}\\%
\begin{array}
[c]{cc}%
i_{D\left\{  n+1\right\}  _{n}\rightarrow D^{n+1}} & \downarrow
\end{array}
&  &
\begin{array}
[c]{cc}%
\Psi_{D^{n+1}} & \downarrow
\end{array}
&  &
\begin{array}
[c]{cc}%
\Psi_{D_{n+1}} & \downarrow
\end{array}
\\
D^{n+1} & \overrightarrow{\Phi_{D^{n+1}}} & D^{n+1}\oplus D & \overrightarrow
{+_{D^{n+1}\rightarrow D_{n+1}}\oplus\mathrm{id}_{D}} & D_{n+1}\oplus D
\end{array}
\label{8.4.1.3}%
\end{equation}
\begin{equation}%
\begin{array}
[c]{ccccc}%
D\left\{  n+1\right\}  _{n} & \underrightarrow{+_{D\left\{  n+1\right\}
_{n}\rightarrow D_{n}}} & D_{n} & \underrightarrow{i_{D_{n}\rightarrow
D_{n+1}}} & D_{n+1}\\%
\begin{array}
[c]{cc}%
i_{D\left\{  n+1\right\}  _{n}\rightarrow D^{n+1}} & \downarrow
\end{array}
&  &
\begin{array}
[c]{cc}%
i_{D_{n}\rightarrow D_{n+1}} & \downarrow
\end{array}
&  &
\begin{array}
[c]{cc}%
\Psi_{D_{n+1}} & \downarrow
\end{array}
\\
D^{n+1} & \overrightarrow{+_{D^{n+1}\rightarrow D_{n+1}}} & D_{n+1} &
\overrightarrow{\Phi_{D_{n+1}}} & D_{n+1}\oplus D
\end{array}
\label{8.4.1.4}%
\end{equation}

\end{proof}

\begin{lemma}
\label{t8.4.2}Let $t\in E\otimes\mathcal{W}_{D}$ and $\gamma\in E\otimes
\mathcal{W}_{D_{n+1}}$ with
\begin{align*}
&  \left(  \mathrm{id}_{E}\otimes\mathcal{W}_{i_{1\rightarrow D}}\right)
\left(  t\right) \\
&  =\left(  \mathrm{id}_{E}\otimes\mathcal{W}_{i_{1\rightarrow D_{n+1}}%
}\right)  \left(  \gamma\right)  \text{.}%
\end{align*}
Then we have
\[
\left(  \mathrm{id}_{E}\otimes\mathcal{W}_{+_{D^{n+1}\rightarrow D_{n+1}}%
}\right)  (t\dot{+}\gamma)=t\dot{+}\left(  \mathrm{id}_{E}\otimes
\mathcal{W}_{+_{D^{n+1}\rightarrow D_{n+1}}}\right)  \left(  \gamma\right)
\]

\end{lemma}

\begin{proof}
This follows simply from a commutative cubical diagram, which is depicted here
separately as the upper square, the lower square (\ref{8.4.1.2}) and the
rounding four side squares:
\[%
\begin{array}
[c]{ccc}%
1 & \underrightarrow{i_{1\rightarrow D}} & D\\%
\begin{array}
[c]{cc}%
\mathrm{id}_{1} & \downarrow
\end{array}
&  &
\begin{array}
[c]{cc}%
\downarrow & \mathrm{id}_{D}%
\end{array}
\\
1 & \overrightarrow{i_{1\rightarrow D}} & D
\end{array}
\]
\[%
\begin{array}
[c]{ccccc}%
1 & \underrightarrow{i_{1\rightarrow D}} & D & \underrightarrow{\mathrm{id}%
_{D}} & D\\%
\begin{array}
[c]{cc}%
i_{1\rightarrow D^{n+1}} & \downarrow
\end{array}
&  &
\begin{array}
[c]{cc}%
\Xi_{D^{n+1}} & \downarrow
\end{array}
&  &
\begin{array}
[c]{cc}%
\Xi_{D_{n+1}} & \downarrow
\end{array}
\\
D^{n+1} & \overrightarrow{\Phi_{D^{n+1}}} & D^{n+1}\oplus D & \overrightarrow
{+_{D^{n+1}\rightarrow D_{n+1}}\oplus\mathrm{id}_{D}} & D_{n+1}\oplus D
\end{array}
\]
\[%
\begin{array}
[c]{ccccc}%
1 & \underrightarrow{\mathrm{id}_{1}} & 1 & \underrightarrow{i_{1\rightarrow
D}} & D\\%
\begin{array}
[c]{cc}%
i_{1\rightarrow D^{n+1}} & \downarrow
\end{array}
&  &
\begin{array}
[c]{cc}%
i_{1\rightarrow D_{n+1}} & \downarrow
\end{array}
&  &
\begin{array}
[c]{cc}%
\Xi_{D_{n+1}} & \downarrow
\end{array}
\\
D^{n+1} & \overrightarrow{+_{D^{n+1}\rightarrow D_{n+1}}} & D_{n+1} &
\overrightarrow{\Phi_{D_{n+1}}} & D_{n+1}\oplus D
\end{array}
\]

\end{proof}

Now we are ready to state the main result of this subsection.

\begin{theorem}
\label{t8.4.10}We have the following:

\begin{enumerate}
\item For any $\nabla^{+},\nabla^{-}\in\mathbb{J}_{x}^{D^{n+1}}\left(
\pi\right)  $ and any $\gamma\in\left(  M\otimes\mathcal{W}_{D_{n+1}}\right)
_{\pi(x)}$ with
\[
\pi_{n+1,n}\left(  \nabla^{+}\right)  =\pi_{n+1,n}\left(  \nabla^{-}\right)
\text{,}%
\]
we have
\begin{align*}
&  \psi_{n+1}(\nabla^{+})(\gamma)\dot{-}\psi_{n+1}(\nabla^{-})(\gamma)\\
&  =\nabla^{+}(\left(  \mathrm{id}_{E}\otimes\mathcal{W}_{+_{D^{n+1}%
\rightarrow D_{n+1}}}\right)  \left(  \gamma\right)  )\dot{-}\nabla
^{-}(\left(  \mathrm{id}_{E}\otimes\mathcal{W}_{+_{D^{n+1}\rightarrow D_{n+1}%
}}\right)  \left(  \gamma\right)  )
\end{align*}

\item For any $\nabla\in\mathbb{J}_{x}^{D^{n+1}}\left(  \pi\right)  $, any
$t\in\left(  M\otimes\mathcal{W}_{D}\right)  _{\pi(x)}$ and any $\gamma
\in\left(  M\otimes\mathcal{W}_{D_{n+1}}\right)  _{\pi(x)}$, we have
\begin{align*}
&  \left(  \mathrm{id}_{E}\otimes\mathcal{W}_{+_{D^{n+1}\rightarrow D_{n+1}}%
}\right)  \left(  \left(  \pi_{n+1,1}\left(  \psi_{n+1}(\nabla)\right)
\right)  \left(  t\right)  \dot{+}\psi_{n+1}(\nabla)(\gamma)\right) \\
&  =\left(  \pi_{n+1,1}\left(  \nabla\right)  \right)  \left(  t\right)
\dot{+}\nabla(\left(  \mathrm{id}_{M}\otimes\mathcal{W}_{+_{D^{n+1}\rightarrow
D_{n+1}}}\right)  \left(  \gamma\right)  )
\end{align*}

\end{enumerate}
\end{theorem}

\begin{proof}
We deal with the two statements separately.

\begin{enumerate}
\item Since
\begin{align*}
&  \nabla^{\pm}(\left(  \mathrm{id}_{M}\otimes\mathcal{W}_{+_{D^{n+1}%
\rightarrow D_{n+1}}}\right)  \left(  \gamma\right)  )\\
&  =\left(  \mathrm{id}_{E}\otimes\mathcal{W}_{+_{D^{n+1}\rightarrow D_{n+1}}%
}\right)  \left(  \left(  \psi_{n+1}(\nabla^{\pm})\right)  (\gamma)\right)
\end{align*}
by the very definition of $\psi_{n+1}(\nabla^{\pm})$, we have
\begin{align*}
&  \nabla^{+}(\left(  \mathrm{id}_{M}\otimes\mathcal{W}_{+_{D^{n+1}\rightarrow
D_{n+1}}}\right)  \left(  \gamma\right)  )\dot{-}\nabla^{-}(\left(
\mathrm{id}_{M}\otimes\mathcal{W}_{+_{D^{n+1}\rightarrow D_{n+1}}}\right)
\left(  \gamma\right)  )\\
&  =\left(  \mathrm{id}_{E}\otimes\mathcal{W}_{+_{D^{n+1}\rightarrow D_{n+1}}%
}\right)  \left(  \left(  \psi_{n+1}(\nabla^{+})\right)  (\gamma)\right)
\dot{-}\left(  \mathrm{id}_{E}\otimes\mathcal{W}_{+_{D^{n+1}\rightarrow
D_{n+1}}}\right)  \left(  \left(  \psi_{n+1}(\nabla^{-})\right)
(\gamma)\right) \\
&  =\psi_{n+1}(\nabla^{+})(\gamma)\dot{-}\psi_{n+1}(\nabla^{-})(\gamma)\text{
\ \ [by Lemma \ref{t8.4.1}]}%
\end{align*}

\item Since
\begin{align*}
&  \nabla(\left(  \mathrm{id}_{M}\otimes\mathcal{W}_{+_{D^{n+1}\rightarrow
D_{n+1}}}\right)  \left(  \gamma\right)  )\\
&  =\left(  \mathrm{id}_{E}\otimes\mathcal{W}_{+_{D^{n+1}\rightarrow D_{n+1}}%
}\right)  \left(  \left(  \psi_{n+1}(\nabla)\right)  (\gamma)\right)
\end{align*}
by the very definition of $\psi_{n+1}(\nabla)$ and
\[
\left(  \pi_{n+1,1}\left(  \nabla\right)  \right)  \left(  t\right)  =\left(
\pi_{n+1,1}\left(  \psi_{n+1}(\nabla)\right)  \right)  \left(  t\right)
\]
by dint of Proposition \ref{t7.1.4}, we have
\begin{align*}
&  \left(  \pi_{n+1,1}\left(  \nabla\right)  \right)  \left(  t\right)
\dot{+}\nabla(\left(  \mathrm{id}_{M}\otimes\mathcal{W}_{+_{D^{n+1}\rightarrow
D_{n+1}}}\right)  \left(  \gamma\right)  )\\
&  =\left(  \pi_{n+1,1}\left(  \nabla\right)  \right)  \left(  t\right)
\dot{+}\left(  \mathrm{id}_{E}\otimes\mathcal{W}_{+_{D^{n+1}\rightarrow
D_{n+1}}}\right)  \left(  \left(  \psi_{n+1}(\nabla)\right)  (\gamma)\right)
\\
&  =\left(  \mathrm{id}_{E}\otimes\mathcal{W}_{+_{D^{n+1}\rightarrow D_{n+1}}%
}\right)  \left(  \left(  \pi_{n+1,1}\left(  \psi_{n+1}(\nabla)\right)
\right)  \left(  t\right)  \dot{+}\psi_{n+1}(\nabla)(\gamma)\right) \\
&  \text{[by Lemma \ref{t8.4.2}]}%
\end{align*}

\end{enumerate}
\end{proof}

Now we would like to discuss the relationship between $\mathbb{S}^{D^{n+1}%
}(\pi)$ and $\mathbb{S}^{D_{n+1}}(\pi)$.

\begin{proposition}
\label{t8.4.4}For any $\omega\in\mathbb{S}_{x}^{D^{n+1}}(\pi)$, the mapping
\[
\gamma\in\left(  M\otimes\mathcal{W}_{D_{n+1}}\right)  _{\pi(x)}\mapsto
\omega\left(  \left(  \mathrm{id}_{M}\otimes\mathcal{W}_{+_{D^{n+1}\rightarrow
D_{n+1}}}\right)  \left(  \gamma\right)  \right)  \text{,}%
\]
denoted by $\phi_{n+1}(\omega)$, belongs to $\mathbb{S}_{x}^{D_{n+1}}(\pi) $,
thereby giving rise to a function $\phi_{n+1}:\mathbb{S}^{D^{n+1}}%
(\pi)\rightarrow\mathbb{S}^{D_{n+1}}(\pi)$.
\end{proposition}

\begin{proof}
For $n=0$, the statement is trivial. For any $\omega\in\mathbb{S}_{x}%
^{D^{n+1}}(\pi)$, there exist $\nabla^{+},\nabla^{-}\in\mathbb{J}_{x}%
^{D^{n+1}}\left(  \pi\right)  $, by dint of Theorem \ref{t8.2.3.4}, such that
\[
\pi_{n+1,n}(\nabla^{+})=\pi_{n+1,n}(\nabla^{-})
\]
and
\[
\omega=\nabla^{+}\dot{-}\nabla^{-}\text{.}%
\]
Then we have the following:

\begin{enumerate}
\item Let $\alpha\in\mathbb{R}$ and $\gamma\in\left(  M\otimes\mathcal{W}%
_{D_{n+1}}\right)  _{\pi(x)}$. Then we have
\begin{align*}
&  \omega\left(  \left(  \mathrm{id}_{M}\otimes\mathcal{W}_{+_{D^{n+1}%
\rightarrow D_{n+1}}}\right)  (\alpha\gamma)\right) \\
&  =\nabla^{+}\left(  \left(  \mathrm{id}_{M}\otimes\mathcal{W}_{+_{D^{n+1}%
\rightarrow D_{n+1}}}\right)  (\alpha\gamma)\right)  \dot{-}\nabla^{-}\left(
\left(  \mathrm{id}_{M}\otimes\mathcal{W}_{+_{D^{n+1}\rightarrow D_{n+1}}%
}\right)  (\alpha\gamma)\right) \\
&  =\psi_{n+1}(\nabla^{+})(\alpha\gamma)\dot{-}\psi_{n+1}(\nabla^{-}%
)(\alpha\gamma)\text{ \ \ \ [by Theorem \ref{t8.4.10}]}\\
&  =\alpha(\psi_{n+1}(\nabla^{+})(\gamma))\dot{-}\alpha(\psi_{n+1}(\nabla
^{-})(\gamma))\\
&  =\alpha^{n+1}(\psi_{n+1}(\nabla^{+})(\gamma)\dot{-}\psi_{n+1}(\nabla
^{-})(\gamma))\\
&  =\alpha^{n+1}(\nabla^{+}(\gamma_{D^{n+1}})\dot{-}\nabla^{-}(\gamma
_{D^{n+1}}))\text{ \ \ \ [by Theorem \ref{t8.4.10}]}\\
&  =\alpha^{n+1}\omega\left(  \left(  \mathrm{id}_{M}\otimes\mathcal{W}%
_{+_{D^{n+1}\rightarrow D_{n+1}}}\right)  \left(  \gamma\right)  \right)
\end{align*}
so that $\phi_{n+1}(\omega)$ abides by the first condition in Definition
\ref{d8.3.2.1}.

\item The proof that the mapping $\phi_{n+1}(\omega)$ abides by the second
condition in Definition \ref{d8.3.2.1}, which is similar to the above, is
safely left to the reader.

\item Let $\rho$ be a simple polynomial of $d\in D_{n+1}$ and $\gamma
\in\left(  M\otimes\mathcal{W}_{D_{l}}\right)  _{\pi(x)}$ with $\mathrm{\dim
}_{n+1}\rho=l<n+1$, we have
\begin{align*}
&  \omega\left(  \left(  \mathrm{id}_{M}\otimes\mathcal{W}_{+_{D^{n+1}%
\rightarrow D_{n+1}}}\right)  \left(  \left(  \mathrm{id}_{M}\otimes
\mathcal{W}_{\rho}\right)  \left(  \gamma\right)  \right)  \right) \\
&  =\nabla^{+}\left(  \left(  \mathrm{id}_{M}\otimes\mathcal{W}_{+_{D^{n+1}%
\rightarrow D_{n+1}}}\right)  \left(  \left(  \mathrm{id}_{M}\otimes
\mathcal{W}_{\rho}\right)  \left(  \gamma\right)  \right)  \right)  \dot{-}\\
&  \nabla^{-}\left(  \left(  \mathrm{id}_{M}\otimes\mathcal{W}_{+_{D^{n+1}%
\rightarrow D_{n+1}}}\right)  \left(  \left(  \mathrm{id}_{M}\otimes
\mathcal{W}_{\rho}\right)  \left(  \gamma\right)  \right)  \right) \\
&  =\left(  \psi_{n+1}(\nabla^{+})\right)  \left(  \left(  \mathrm{id}%
_{M}\otimes\mathcal{W}_{\rho}\right)  \left(  \gamma\right)  \right)  \dot
{-}\left(  \psi_{n+1}(\nabla^{-})\right)  \left(  \left(  \mathrm{id}%
_{M}\otimes\mathcal{W}_{\rho}\right)  \left(  \gamma\right)  \right) \\
&  \text{[by Theorem \ref{t8.4.10}]}\\
&  =\left(  \mathrm{id}_{E}\otimes\mathcal{W}_{\rho}\right)  (\left(
\pi_{n+1,l}(\psi_{n+1}(\nabla^{+}))\right)  (\gamma))\dot{-}\left(
\mathrm{id}_{E}\otimes\mathcal{W}_{\rho}\right)  (\left(  \pi_{n+1,l}%
(\psi_{n+1}(\nabla^{-}))\right)  (\gamma))\\
&  =\left(  \mathrm{id}_{E}\otimes\mathcal{W}_{\rho}\right)  (\left(  \psi
_{l}(\pi_{n+1,l}(\nabla^{+}))\right)  (\gamma))\dot{-}\left(  \mathrm{id}%
_{E}\otimes\mathcal{W}_{\rho}\right)  (\left(  \psi_{l}(\pi_{n+1,l}(\nabla
^{-}))\right)  (\gamma))\\
&  \text{[by Proposition \ref{t7.1.4}]}\\
&  =0\text{,}%
\end{align*}
so that $\phi_{n+1}(\omega)$ abides by the third condition in Definition
\ref{d8.3.2.1}.
\end{enumerate}
\end{proof}

Let us fix our terminology.

\begin{definition}
Given an affine bundle $\pi_{1}:E_{1}\rightarrow M_{1}$ over a vector bundle
$\xi_{1}:P_{1}\rightarrow M_{1}$ and another affine bundle $\pi_{2}%
:E_{2}\rightarrow M_{2}$ over another vector bundle $\xi_{2}:P_{2}\rightarrow
M_{2}$, a triple $(f,g,h)$ of mappings $f:M_{1}\rightarrow M_{2}$,
$g:E_{1}\rightarrow E_{2}$ and $h:P_{1}\rightarrow P_{2}$ is called a
\textit{morphism of affine bundles} from the affine bundle $\pi_{1}%
:E_{1}\rightarrow M_{1}$ over the vector bundle $\xi_{1}:P_{1}\rightarrow
M_{1}$ to the affine bundle $\pi_{2}:E_{2}\rightarrow E_{2}$ over the vector
bundle $\xi_{2}:P_{2}\rightarrow M_{2}$ provided that they abide by the
following three conditions:

\begin{enumerate}
\item $(f,g)$ is a morphism of bundles from $\pi_{1}$ to $\pi_{2}$. In other
words, the following diagram is commutative:
\[%
\begin{array}
[c]{ccc}%
E_{1} & \underrightarrow{g} & E_{2}\\
\pi_{1}\downarrow &  & \downarrow\pi_{2}\\
M_{1} & \overrightarrow{f} & E_{2}%
\end{array}
\]

\item $(f,h)$ is a morphism of bundles from $\xi_{1}$ to $\xi_{2}$. In other
words, the following diagram is commutative:
\[%
\begin{array}
[c]{ccc}%
P_{1} & \underrightarrow{h} & P_{2}\\
\xi_{1}\downarrow &  & \downarrow\xi_{2}\\
M_{1} & \overrightarrow{f} & E_{2}%
\end{array}
\]

\item For any $x\in M_{1}$, $(g\mid_{E_{1,x}},h\mid_{P_{1,x}})$ is a morphism
of affine spaces from $(E_{1,x},P_{1,x})$ to $(E_{2,x},P_{2,x})$.
\end{enumerate}
\end{definition}

Using this terminology, we can summarize Theorem \ref{t8.4.10} succinctly as follows:

\begin{theorem}
\label{t8.4.5}The triple $(\psi_{n},\psi_{n+1},\phi_{n+1}\times\psi_{n}) $ of
mappings is a morphism of affine bundles from the affine bundle $\pi
_{n+1,n}:\mathbb{J}^{D^{n+1}}(\pi)\rightarrow\mathbb{J}^{D^{n}}(\pi)$ over the
vector bundle $\mathbb{S}^{D^{n+1}}(\pi)\underset{E}{\times}\mathbb{J}^{D^{n}%
}(\pi)\rightarrow\mathbb{J}^{D^{n}}(\pi)$ in Theorem \ref{t8.2.3.4} to the
affine bundle $\pi_{n+1,n}:\mathbb{J}^{D_{n+1}}(\pi)\rightarrow\mathbb{J}%
^{D_{n}}(\pi)$ over the vector bundle $\mathbb{S}^{D_{n+1}}(\pi)\underset
{E}{\times}\mathbb{J}^{D_{n}}(\pi)\rightarrow\mathbb{J}^{D_{n}}(\pi)$ in
Theorem \ref{t8.3.3.3}.
\end{theorem}

\section{\label{s2}Taylor Representations}

The following results are merely special cases of the general Taylor-type
theorem such as seen in \cite{kock2} (Part III, Proposition 5.2).

\begin{proposition}
\label{t2.1}Any $f\in\mathbb{R}^{p}\otimes\mathcal{W}_{D_{n}}$ is of a unique
Taylor representation of the form
\[
\delta\in\mathbb{R}\longmapsto(x^{i})+\delta(y_{1}^{i})+\delta^{2}(y_{2}%
^{i})+...+\delta^{n}(y_{n}^{i})\in\mathbb{R}^{p}%
\]
with $(x^{i}),(y_{1}^{i}),(y_{2}^{i}),...,(y_{n}^{i})\in\mathbb{R}^{p}$.
\end{proposition}

\begin{proposition}
\label{t2.2}Any $f\in\mathbb{R}^{p}\otimes\mathcal{W}_{D^{n}}$ is of a unique
Taylor representation of the form
\[
(\delta_{1},...,\delta_{n})\in\mathbb{R}^{n}\longmapsto(x^{i})+\sum_{r=1}%
^{n}\sum_{1\leq k_{1}<...<k_{r}\leq n}\delta_{k_{1}}...\delta_{k_{r}}%
(y_{k_{1},...,k_{r}}^{i})\in\mathbb{R}^{p}%
\]
with $(x^{i}),(y_{k_{1},...,k_{r}}^{i})\in\mathbb{R}^{p}$.
\end{proposition}

\begin{proposition}
\label{t2.3}Any $f\in\mathbb{R}^{p}\otimes\mathcal{W}_{D\left(  n\right)
_{m}}$ is of a unique Taylor representation of the form
\[
(\delta_{1},...,\delta_{n})\in\mathbb{R}^{n}\longmapsto(x^{i})+\sum_{r=1}%
^{m}\sum_{1\leq k_{1}\leq...\leq k_{r}\leq n}\delta_{k_{1}}...\delta_{k_{r}%
}(y_{k_{1},...,k_{r}}^{i})\in\mathbb{R}^{p}%
\]
with $(x^{i}),(y_{k_{1},...,k_{r}}^{i})\in\mathbb{R}^{p}$.
\end{proposition}

\section{\label{s3}The Basic Framework with Coordinates}

This section is inspired much by \cite{kock1}.

\subsection{\label{s3.1}The Basic Framework}

\begin{notation}
We denote by $\mathcal{J}^{n}(\pi)$\ the totality of
\[
\gamma\in E\otimes\mathcal{W}_{D(p)_{n}}%
\]
such that
\[
\left(  \pi\otimes\mathrm{id}_{\mathcal{W}_{D(p)_{n}}}\right)  \left(
\gamma\right)  \in M\otimes\mathcal{W}_{D(p)_{n}}%
\]
is degenerate in the sense that
\begin{align*}
&  \left(  \pi\otimes\mathrm{id}_{\mathcal{W}_{D(p)_{n}}}\right)  \left(
\gamma\right) \\
&  =\left(  \mathrm{id}_{M}\otimes\mathcal{W}_{D(p)_{n}\rightarrow1}\right)
\left(  \gamma^{\prime}\right)
\end{align*}
for some $\gamma^{\prime}\in M\otimes\mathcal{W}_{1}=M$.
\end{notation}

\begin{notation}
We denote by $\mathcal{S}^{n}(\pi)$\ the totality of
\[
t\in E\otimes\mathcal{W}_{D\left(  _{p+n}C_{n+1}\right)  }%
\]
such that
\[
\left(  \pi\otimes\mathrm{id}_{\mathcal{W}_{D\left(  _{p+n}C_{n+1}\right)  }%
}\right)  \left(  t\right)  \in M\otimes\mathcal{W}_{D\left(  _{p+n}%
C_{n+1}\right)  }%
\]
is degenerate in the sense that
\begin{align*}
&  \left(  \pi\otimes\mathrm{id}_{\mathcal{W}_{D\left(  _{p+n}C_{n+1}\right)
}}\right)  \left(  t\right) \\
&  =\left(  \mathrm{id}_{M}\otimes\mathcal{W}_{D\left(  _{p+n}C_{n+1}\right)
\rightarrow1}\right)  \left(  t^{\prime}\right)
\end{align*}
for some $t^{\prime}\in M\otimes\mathcal{W}_{1}=M$.
\end{notation}

\begin{remark}
\begin{enumerate}
\item Each $\gamma\in E\otimes\mathcal{W}_{D(p)_{n}}$ can be identified
unquely with a sequence
\[
\left(  x^{1},...,x^{p},u^{1},...,u^{q},x_{i_{1}}^{i},x_{i_{1},i_{2}}%
^{i},...,x_{i_{1},i_{2},...,i_{n}}^{i},u_{i_{1}}^{j},u_{i_{1},i_{2}}%
^{j},...,u_{i_{1},i_{2},...,i_{n}}^{j}\right)  _{1\leq i_{1}\leq i_{2}%
\leq...\leq i_{n}\leq p}%
\]
of real numbers at length
\[
p+q+\left(  p+q\right)  p+...+\left(  p+q\right)  _{p+n-1}C_{n}%
\]
in the sense that the Taylor representation of $\gamma$ is
\begin{align*}
\left(  \delta_{1},...,\delta_{p}\right)   &  \in\mathbb{R}^{p}\mapsto\left(
x^{1},...,x^{p},u^{1},...,u^{q}\right)  +\sum_{1\leq i_{1}\leq p}\left(
x_{i_{1}}^{1},...,x_{i_{1}}^{p},u_{i_{1}}^{1},...,u_{i_{1}}^{q}\right)
\delta_{i_{1}}\\
&  +\sum_{1\leq i_{1}\leq i_{2}\leq p}\left(  x_{i_{1},i_{2}}^{1}%
,...,x_{i_{1},i_{2}}^{p},u_{i_{1},i_{2}}^{1},...,u_{i_{1},i_{2}}^{q}\right)
\delta_{i_{1}}\delta_{i_{2}}+...\\
&  +\sum_{1\leq i_{1}\leq i_{2}\leq...\leq i_{n}\leq p}\left(  x_{i_{1}%
,i_{2},...,i_{n}}^{1},...,x_{i_{1},i_{2},...,i_{n}}^{p},u_{i_{1}%
,i_{2},...,i_{n}}^{1},...,u_{i_{1},i_{2},...,i_{n}}^{q}\right)  \delta_{i_{1}%
}\delta_{i_{2}}...\delta_{i_{n}}\\
&  \in\mathbb{R}^{p+q}%
\end{align*}

\item Each $\nabla\in\mathcal{J}^{n}(\pi)$ can be identified uniquely with a
sequence
\[
\left(  x^{1},...,x^{p},u^{1},...,u^{q},u_{i_{1}}^{j},u_{i_{1},i_{2}}%
^{j},...,u_{i_{1},i_{2},...,i_{n}}^{j}\right)  _{1\leq i_{1}\leq i_{2}%
\leq...\leq i_{n}\leq p}%
\]
of real numbers at length
\[
p+q+qp+...+q_{p+n-1}C_{n}%
\]
\ in the sense that the Taylor representation of $\nabla$ is
\begin{align*}
\left(  \delta_{1},...,\delta_{p}\right)   &  \in\mathbb{R}^{p}\mapsto\left(
x^{1},...,x^{p},u^{1},...,u^{q}\right)  +\sum_{1\leq i_{1}\leq p}\left(
0,...0,u_{i_{1}}^{1},...,u_{i_{1}}^{q}\right)  \delta_{i_{1}}\\
&  +\sum_{1\leq i_{1}\leq i_{2}\leq p}\left(  0,...0,u_{i_{1},i_{2}}%
^{1},...,u_{i_{1},i_{2}}^{q}\right)  \delta_{i_{1}}\delta_{i_{2}}+...\\
&  +\sum_{1\leq i_{1}\leq i_{2}\leq...\leq i_{n}\leq p}\left(  0,...0,u_{i_{1}%
,i_{2},...,i_{n}}^{1},...,u_{i_{1},i_{2},...,i_{n}}^{q}\right)  \delta_{i_{1}%
}\delta_{i_{2}}...\delta_{i_{n}}\\
&  \in\mathbb{R}^{p+q}%
\end{align*}

\item Each $t\in E\otimes\mathcal{W}_{D\left(  _{p+n}C_{n+1}\right)  }$ can be
identified unquely with a sequence
\[
\left(  x^{1},...,x^{p},u^{1},...,u^{q},x_{i_{1},i_{2},...,i_{n+1}}%
^{1},...,x_{i_{1},i_{2},...,i_{n+1}}^{p},u_{i_{1},i_{2},...,i_{n+1}}%
^{1},...,u_{i_{1},i_{2},...,i_{n+1}}^{q}\right)  _{1\leq i_{1}\leq i_{2}%
\leq...\leq i_{n+1}\leq p}%
\]
of real numbers at length
\[
p+q+\left(  p+q\right)  _{p+n}C_{n+1}%
\]
in the sense that the Taylor representation of $t$ is
\begin{align*}
&  \left(  \delta_{i_{1},i_{2},...,i_{n+1}}\right)  _{1\leq i_{1}\leq
i_{2}\leq...\leq i_{n+1}\leq p}\\
&  \in\mathbb{R}^{\left(  _{p+n}C_{n+1}\right)  }\mapsto\left(  x^{1}%
,...,x^{p},u^{1},...,u^{q}\right) \\
&  +\sum_{1\leq i_{1}\leq i_{2}\leq...\leq i_{n+1}\leq p}\left(
x_{i_{1},i_{2},...,i_{n+1}}^{1},...,x_{i_{1},i_{2},...,i_{n+1}}^{p}%
,u_{i_{1},i_{2},...,i_{n+1}}^{1},...,u_{i_{1},i_{2},...,i_{n+1}}^{q}\right)
\delta_{i_{1},i_{2},...,i_{n+1}}\\
&  \in\mathbb{R}^{p+q}%
\end{align*}

\item Each $\omega\in\mathcal{S}^{n+1}(\pi)$ can be identified unquely with a
sequence
\[
\left(  x^{1},...,x^{p},u^{1},...,u^{q},u_{i_{1},i_{2},...,i_{n+1}}%
^{1},...,u_{i_{1},i_{2},...,i_{n+1}}^{q}\right)  _{1\leq i_{1}\leq i_{2}%
\leq...\leq i_{n+1}\leq p}%
\]
of real numbers at length
\[
p+q+q_{p+n}C_{n+1}%
\]
in the sense that the Taylor representation of $\omega$ is
\begin{align*}
&  \left(  \delta_{i_{1},i_{2},...,i_{n+1}}\right)  _{1\leq i_{1}\leq
i_{2}\leq...\leq i_{n+1}\leq p}\\
&  \in\mathbb{R}^{\left(  _{p+n}C_{n+1}\right)  }\mapsto\left(  x^{1}%
,...,x^{p},u^{1},...,u^{q}\right)  +\\
&  \sum_{1\leq i_{1}\leq i_{2}\leq...\leq i_{n+1}\leq p}\left(
0,...0,u_{i_{1},i_{2},...,i_{n+1}}^{1},...,u_{i_{1},i_{2},...,i_{n+1}}%
^{q}\right)  \delta_{i_{1},i_{2},...,i_{n+1}}\in\mathbb{R}^{p+q}%
\end{align*}

\end{enumerate}
\end{remark}

\begin{notation}
Since $D(p)_{n}\subset D(p)_{n+1}$, there is a canonical projection
$E\otimes\mathcal{W}_{D(p)_{n+1}}\rightarrow E\otimes\mathcal{W}_{D(p)_{n}}$,
which restricts itself naturally to $\mathcal{J}^{n+1}(\pi)\rightarrow
\mathcal{J}^{n}(\pi)$. Both of them are denoted by $\pi_{n+1,n}$. We have
\begin{align*}
&  \pi_{n+1,n}\left(  x^{1},...,x^{p},u^{1},...,u^{q},x_{i_{1}}^{i}%
,...,x_{i_{1},i_{2},...,i_{n}}^{j},x_{i_{1},i_{2},...,i_{n+1}}^{i},u_{i_{1}%
}^{j},...,u_{i_{1},i_{2},...,i_{n}}^{j},u_{i_{1},i_{2},...,i_{n+1}}^{j}\right)
\\
&  =\left(  x^{1},...,x^{p},u^{1},...,u^{q},x_{i_{1}}^{i},x_{i_{1},i_{2}}%
^{i},...,x_{i_{1},i_{2},...,i_{n}}^{i},u_{i_{1}}^{j},u_{i_{1},i_{2}}%
^{j},...,u_{i_{1},i_{2},...,i_{n}}^{j}\right)
\end{align*}
and
\begin{align*}
&  \pi_{n+1,n}\left(  x^{1},...,x^{p},u^{1},...,u^{q},u_{i_{1}}^{j}%
,...,u_{i_{1},i_{2},...,i_{n}}^{j},u_{i_{1},i_{2},...,i_{n+1}}^{j}\right) \\
&  =\left(  x^{1},...,x^{p},u^{1},...,u^{q},u_{i_{1}}^{j},...,u_{i_{1}%
,i_{2},...,i_{n}}^{j}\right)
\end{align*}

\end{notation}

\subsection{\label{s3.2}The Affine Bundle Theorem within the Basic Framework}

It is easy to see that

\begin{lemma}
\label{t3.2.1}The diagram
\[%
\begin{array}
[c]{ccc}%
D(p)_{n} & \underrightarrow{i_{D(p)_{n}\rightarrow D(p)_{n+1}}} & D(p)_{n+1}\\
i_{D(p)_{n}\rightarrow D(p)_{n+1}}\downarrow &  & \downarrow\Psi_{D(p)_{n+1}%
}\\
D(p)_{n+1} & \overrightarrow{\Phi_{D(p)_{n+1}}} & D(p)_{n+1}\oplus D\left(
_{p+n}C_{n+1}\right)
\end{array}
\]
is a quasi-colimit diagram, where $\Phi_{D(p)_{n+1}}$ is the canonical
injection of $D(p)_{n+1}$ into $D(p)_{n+1}\oplus D\left(  _{p+n}%
C_{n+1}\right)  $, and $\Psi_{D(p)_{n+1}}$ is the mapping
\[
\left(  d_{1},...,d_{p}\right)  \in D(p)_{n+1}\mapsto\left(  d_{1}%
,...,d_{p},d_{1}^{n+1},d_{1}^{n}d_{2},...\right)  \in D(p)_{n+1}\oplus
D\left(  _{p+n}C_{n+1}\right)
\]
with the sequence $d_{1}^{n+1},d_{1}^{n}d_{2},...$ being that of $d_{k_{1}%
}d_{k_{2}}...d_{k_{n+1}}$'s $\left(  1\leq k_{1}\leq k_{2}\leq...\leq
k_{n+1}\leq n+1\right)  $ in lexicographic order.
\end{lemma}

This implies at once that

\begin{proposition}
\label{t3.2.2}Given $\gamma_{+},\gamma_{-}\in$ $E\otimes\mathcal{W}%
_{D(p)_{n+1}}$ with
\[
\left(  \mathrm{id}_{E}\otimes\mathcal{W}_{i_{D(p)_{n}\rightarrow D(p)_{n+1}}%
}\right)  \left(  \gamma_{+}\right)  =\left(  \mathrm{id}_{E}\otimes
\mathcal{W}_{i_{D(p)_{n}\rightarrow D(p)_{n+1}}}\right)  \left(  \gamma
_{-}\right)  \text{,}%
\]
there exists unique $\gamma\in E\otimes\mathcal{W}_{D(p)_{n+1}\oplus D\left(
_{p+n}C_{n+1}\right)  }$ with
\begin{align*}
\left(  \mathrm{id}_{E}\otimes\mathcal{W}_{\Psi_{D(p)_{n+1}}}\right)  \left(
\gamma\right)   &  =\gamma_{+}\text{ and}\\
\left(  \mathrm{id}_{E}\otimes\mathcal{W}_{\Phi_{D(p)_{n+1}}}\right)  \left(
\gamma\right)   &  =\gamma_{-}%
\end{align*}

\end{proposition}

\begin{notation}
Under the same notation as in the above proposition, we denote
\[
\left(  \mathrm{id}_{E}\otimes\mathcal{W}_{\Xi_{D(p)_{n+1}}}\right)  \left(
\gamma\right)
\]
by
\[
\gamma_{+}\dot{-}\gamma_{-}\text{,}%
\]
where $\Xi_{D^{n}}:D\left(  _{p+n}C_{n+1}\right)  \rightarrow D(p)_{n+1}\oplus
D\left(  _{p+n}C_{n+1}\right)  $ is the canonical injection.
\end{notation}

\begin{remark}
Given
\begin{align*}
\gamma_{+}  &  =\left(  x^{1},...,x^{p},u^{1},...,u^{q},x_{i_{1}}^{i}%
,x_{i_{1},i_{2}}^{i},...,x_{i_{1},i_{2},...,i_{n+1}}^{i},u_{i_{1}}%
^{j},u_{i_{1},i_{2}}^{j},...,u_{i_{1},i_{2},...,i_{n+1}}^{j}\right)  ,\\
\gamma_{-}  &  =\left(  y^{1},...,y^{p},v^{1},...,v^{q},y_{i_{1}}^{i}%
,y_{i_{1},i_{2}}^{i},...,y_{i_{1},i_{2},...,i_{n+1}}^{i},v_{i_{1}}%
^{j},v_{i_{1},i_{2}}^{j},...,v_{i_{1},i_{2},...,i_{n+1}}^{j}\right) \\
&  \in E\otimes\mathcal{W}_{D(p)_{n+1}}\text{,}%
\end{align*}
we have
\[
\left(  \mathrm{id}_{E}\otimes\mathcal{W}_{i_{D(p)_{n}\rightarrow D(p)_{n+1}}%
}\right)  \left(  \gamma_{+}\right)  =\left(  \mathrm{id}_{E}\otimes
\mathcal{W}_{i_{D(p)_{n}\rightarrow D(p)_{n+1}}}\right)  \left(  \gamma
_{-}\right)
\]
iff
\begin{align*}
&  \left(  x^{1},...,x^{p},u^{1},...,u^{q},x_{i_{1}}^{i},x_{i_{1},i_{2}}%
^{i},...,x_{i_{1},i_{2},...,i_{n}}^{i},u_{i_{1}}^{j},u_{i_{1},i_{2}}%
^{j},...,u_{i_{1},i_{2},...,i_{n}}^{j}\right) \\
&  =\left(  y^{1},...,y^{p},v^{1},...,v^{q},y_{i_{1}}^{i},y_{i_{1},i_{2}}%
^{i},...,y_{i_{1},i_{2},...,i_{n}}^{i},v_{i_{1}}^{j},v_{i_{1},i_{2}}%
^{j},...,v_{i_{1},i_{2},...,i_{n}}^{j}\right)  \text{,}%
\end{align*}
in which we get
\begin{align*}
&  \gamma_{+}\dot{-}\gamma_{-}\\
&  =\left(  x^{1},...,x^{p},u^{1},...,u^{q},u_{i_{1},i_{2},...,i_{n+1}}%
^{1}-v_{i_{1},i_{2},...,i_{n+1}}^{1},...,u_{i_{1},i_{2},...,i_{n+1}}%
^{q}-v_{i_{1},i_{2},...,i_{n+1}}^{q}\right)
\end{align*}

\end{remark}

From the very definition of $\dot{-}$, we have

\begin{proposition}
\label{t3.2.3}Let $F$ be a mapping of $E$\ into $E^{\prime}$. Given
$\gamma_{+},\gamma_{-}\in E\otimes\mathcal{W}_{D(p)_{n+1}}$ with
\[
\left(  \mathrm{id}_{E}\otimes\mathcal{W}_{i_{D(p)_{n}\rightarrow D(p)_{n+1}}%
}\right)  \left(  \gamma_{+}\right)  =\left(  \mathrm{id}_{E}\otimes
\mathcal{W}_{i_{D(p)_{n}\rightarrow D(p)_{n+1}}}\right)  \left(  \gamma
_{-}\right)  \text{,}%
\]
we have
\begin{align*}
&  \left(  \mathrm{id}_{E^{\prime}}\otimes\mathcal{W}_{i_{D(p)_{n}\rightarrow
D(p)_{n+1}}}\right)  \left(  \left(  F\otimes\mathrm{id}_{\mathcal{W}%
_{D(p)_{n+1}}}\right)  \left(  \gamma_{+}\right)  \right) \\
&  =\left(  \mathrm{id}_{E^{\prime}}\otimes\mathcal{W}_{i_{D(p)_{n}\rightarrow
D(p)_{n+1}}}\right)  \left(  \left(  F\otimes\mathrm{id}_{\mathcal{W}%
_{D(p)_{n+1}}}\right)  \left(  \gamma_{-}\right)  \right)
\end{align*}
and
\begin{align*}
&  \left(  F\otimes\mathrm{id}_{\mathcal{W}_{D\left(  _{p+n}C_{n+1}\right)  }%
}\right)  \left(  \gamma_{+}\dot{-}\gamma_{-}\right) \\
&  =\left(  F\otimes\mathrm{id}_{\mathcal{W}_{D(p)_{n+1}}}\right)  \left(
\gamma_{+}\right)  \dot{-}\left(  F\otimes\mathrm{id}_{\mathcal{W}%
_{D(p)_{n+1}}}\right)  \left(  \gamma_{-}\right)
\end{align*}

\end{proposition}

\begin{lemma}
\label{t3.2.4}The diagram
\[%
\begin{array}
[c]{ccc}%
1 & \underrightarrow{i_{1\rightarrow D\left(  _{p+n}C_{n+1}\right)  }} &
D\left(  _{p+n}C_{n+1}\right) \\
i_{1\rightarrow D(p)_{n+1}}\downarrow &  & \downarrow\Xi_{D(p)_{n+1}}\\
D(p)_{n+1} & \overrightarrow{\Phi_{D(p)_{n+1}}} & D(p)_{n+1}\oplus D\left(
_{p+n}C_{n+1}\right)
\end{array}
\]
is a quasi-colimit diagram.
\end{lemma}

This implies at once that

\begin{proposition}
\label{t3.2.5}Given $t\in E\otimes\mathcal{W}_{D\left(  _{p+n}C_{n+1}\right)
}\ $and $\gamma\in E\otimes\mathcal{W}_{D(p)_{n+1}}\ $with
\[
\left(  \mathrm{id}_{E}\otimes\mathcal{W}_{i_{1\rightarrow D\left(
_{p+n}C_{n+1}\right)  }}\right)  \left(  t\right)  =\left(  \mathrm{id}%
_{E}\otimes\mathcal{W}_{i_{1\rightarrow D(p)_{n+1}}}\right)  \left(
\gamma\right)  \text{,}%
\]
there exists unique $\gamma^{\prime}\in E\otimes\mathcal{W}_{D(p)_{n+1}\oplus
D\left(  _{p+n}C_{n+1}\right)  }\ $with
\begin{align*}
\left(  \mathrm{id}_{E}\otimes\mathcal{W}_{\Xi_{D(p)_{n+1}}}\right)  \left(
\gamma^{\prime}\right)   &  =t\text{ and}\\
\left(  \mathrm{id}_{E}\otimes\mathcal{W}_{\Phi_{D(p)_{n+1}}}\right)  \left(
\gamma^{\prime}\right)   &  =\gamma\text{.}%
\end{align*}

\end{proposition}

\begin{notation}
Under the same notation as in the above proposition, we denote
\[
\left(  \mathrm{id}_{E}\otimes\mathcal{W}_{\Psi_{D(p)_{n+1}}}\right)  \left(
\gamma^{\prime}\right)
\]
by $t\dot{+}\gamma$.
\end{notation}

\begin{remark}
Given
\begin{align*}
\gamma &  =\left(  x^{1},...,x^{p},u^{1},...,u^{q},x_{i_{1}}^{i}%
,...,x_{i_{1},i_{2},...,i_{n+1}}^{i},u_{i_{1}}^{j},...,u_{i_{1},i_{2}%
,...,i_{n+1}}^{j}\right) \\
&  \in E\otimes\mathcal{W}_{D(p)_{n+1}}%
\end{align*}
and
\begin{align*}
t  &  =\left(  y^{1},...,y^{p},v^{1},...,v^{q},y_{i_{1},i_{2},...,i_{n+1}}%
^{1},...,y_{i_{1},i_{2},...,i_{n+1}}^{p},v_{i_{1},i_{2},...,i_{n+1}}%
^{1},...,v_{i_{1},i_{2},...,i_{n+1}}^{q}\right) \\
&  \in E\otimes\mathcal{W}_{D\left(  _{p+n}C_{n+1}\right)  }\text{,}%
\end{align*}
we have
\[
\left(  \mathrm{id}_{E}\otimes\mathcal{W}_{i_{1\rightarrow D\left(
_{p+n}C_{n+1}\right)  }}\right)  \left(  t\right)  =\left(  \mathrm{id}%
_{E}\otimes\mathcal{W}_{i_{1\rightarrow D(p)_{n+1}}}\right)  \left(
\gamma\right)
\]
iff
\[
\left(  x^{1},...,x^{p},u^{1},...,u^{q}\right)  =\left(  y^{1},...,y^{p}%
,v^{1},...,v^{q}\right)  \text{,}%
\]
in which we get
\begin{align*}
&  t\dot{+}\gamma\\
&  =\left(  x^{i},u^{j},x_{i_{1}}^{i},...,x_{i_{1},i_{2},...,i_{n}}%
^{i},x_{i_{1},i_{2},...,i_{n+1}}^{i}+y_{i_{1},i_{2},...,i_{n+1}}^{i},u_{i_{1}%
}^{j},...,u_{i_{1},i_{2},...,i_{n}}^{j},u_{i_{1},i_{2},...,i_{n+1}}%
^{j}+v_{i_{1},i_{2},...,i_{n+1}}^{j}\right)
\end{align*}

\end{remark}

From the very definition of $\dot{+}$, we have

\begin{proposition}
\label{t3.2.6}Let $F$ be a mapping of $E$\ into $E^{\prime}$. Given $t\in
E\otimes\mathcal{W}_{D\left(  _{p+n}C_{n+1}\right)  }\ $and $\gamma\in
E\otimes\mathcal{W}_{D(p)_{n+1}}\ $with
\[
\left(  \mathrm{id}_{E}\otimes\mathcal{W}_{i_{1\rightarrow D\left(
_{p+n}C_{n+1}\right)  }}\right)  \left(  t\right)  =\left(  \mathrm{id}%
_{E}\otimes\mathcal{W}_{i_{1\rightarrow D(p)_{n+1}}}\right)  \left(
\gamma\right)  \text{,}%
\]
we have
\begin{align*}
&  \left(  \mathrm{id}_{E^{\prime}}\otimes\mathcal{W}_{i_{1\rightarrow
D\left(  _{p+n}C_{n+1}\right)  }}\right)  \left(  \left(  F\otimes
\mathrm{id}_{\mathcal{W}_{D\left(  _{p+n}C_{n+1}\right)  }}\right)  \left(
t\right)  \right) \\
&  =\left(  \mathrm{id}_{E^{\prime}}\otimes\mathcal{W}_{i_{1\rightarrow
D(p)_{n+1}}}\right)  \left(  \left(  F\otimes\mathrm{id}_{\mathcal{W}%
_{D(p)_{n+1}}}\right)  \left(  \gamma\right)  \right)
\end{align*}
and
\[
\left(  F\otimes\mathrm{id}_{\mathcal{W}_{D(p)_{n+1}}}\right)  \left(
t\dot{+}\gamma\right)  =\left(  F\otimes\mathrm{id}_{\mathcal{W}_{D\left(
_{p+n}C_{n+1}\right)  }}\right)  \left(  t\right)  \dot{+}\left(
F\otimes\mathrm{id}_{\mathcal{W}_{D(p)_{n+1}}}\right)  \left(  \gamma\right)
\text{,}%
\]

\end{proposition}

Now we have

\begin{theorem}
\label{t3.2.7}The canonical projection $\mathrm{id}_{E}\otimes\mathcal{W}%
_{i_{D(p)_{n}\rightarrow D(p)_{n+1}}}:E\otimes\mathcal{W}_{D(p)_{n+1}%
}\mathcal{\rightarrow}E\otimes\mathcal{W}_{D(p)_{n}}\ $is an affine bundle
over the vector bundle $\left(  E\otimes\mathcal{W}_{D\left(  _{p+n}%
C_{n+1}\right)  }\right)  ^{\perp}\underset{E}{\times}\left(  E\otimes
\mathcal{W}_{D(p)_{n}}\right)  \rightarrow E\otimes\mathcal{W}_{D(p)_{n}}$.
\end{theorem}

\begin{theorem}
\label{t3.2.8}The mapping $\pi_{n+1,n}:\mathcal{J}^{n+1}(\pi)\rightarrow
\mathcal{J}^{n}(\pi)$ is an affine bundle over the vector bundle
$\mathcal{S}^{n+1}(\pi)\underset{E}{\times}\mathcal{J}^{n}(\pi)\rightarrow
\mathcal{J}^{n}(\pi)$.
\end{theorem}

\section{\label{s4}The First Approach with Coordinates}

\begin{definition}
We define $\theta_{\mathbf{J}^{1}(\pi)}^{\mathcal{J}^{1}(\pi)}:\mathcal{J}%
^{1}(\pi)\rightarrow\mathbf{J}^{1}(\pi)$ to be
\begin{align*}
&  \theta_{\mathbf{J}^{1}(\pi)}^{\mathcal{J}^{1}(\pi)}\left(  x^{1}%
,...,x^{p},u^{1},...,u^{q},u_{i}^{j}\right) \\
&  =\left[  \delta\in\mathbb{R}\mapsto\left(  x^{1},...,x^{p}\right)  +\left(
y^{1},...,y^{p}\right)  \delta\in\mathbb{R}^{p}\right]  \in\left(
M\otimes\mathcal{W}_{D}\right)  _{\left(  x^{1},...,x^{p}\right)  }\mapsto\\
&  \left[  \delta\in\mathbb{R}\mapsto\left(  x^{1},...,x^{p},u^{1}%
,...,u^{q}\right)  +\left(  y^{1},...,y^{p},\sum_{i=1}^{p}u_{i}^{j}%
y^{i}\right)  \delta\in\mathbb{R}^{p+q}\right] \\
&  \in\left(  E\otimes\mathcal{W}_{D}\right)  _{\left(  x^{1},...,x^{p}%
,u^{1},...,u^{q}\right)  }%
\end{align*}

\end{definition}

\begin{remark}
It is easy to see that the right-hand side of the above formula belongs to
$\mathbf{J}^{1}(\pi)$.
\end{remark}

\begin{theorem}
\label{t4.1}The mapping $\theta_{\mathbf{J}^{1}(\pi)}^{\mathcal{J}^{1}(\pi
)}:\mathcal{J}^{1}(\pi)\rightarrow\mathbf{J}^{1}(\pi)\ $is bijective.
\end{theorem}

\begin{remark}
This gives a coordinate description of $\mathbf{J}^{1}(\pi)$.
\end{remark}

Now we are going to consider $\widetilde{\mathbf{J}}^{2}(\pi)=\mathbf{J}%
^{1}(\pi_{1})$, which has a coordinate description as follows:%

\begin{align*}
&  \theta_{\mathbf{J}^{1}(\pi_{1})}^{\mathcal{J}^{1}(\pi_{1})}\left(
x^{i},u^{j},u_{i_{1}}^{j},u_{;i_{2}}^{j},u_{i_{1};i_{2}}^{j}\right) \\
&  =\left[  \delta\in\mathbb{R}\mapsto\left(  x^{i}\right)  +\left(
y^{i}\right)  \delta\in\mathbb{R}^{p}\right]  \in\left(  M\otimes
\mathcal{W}_{D}\right)  _{\left(  x^{i}\right)  }\mapsto\\
&  \left[  \delta\in\mathbb{R}\mapsto\left(  x^{i},u^{j},u_{i_{1}}^{j}\right)
+\left(  y^{i},\sum_{i_{2}=1}^{p}u_{;i_{2}}^{j}y^{i_{2}},\sum_{i_{2}=1}%
^{p}u_{i_{1};i_{2}}^{j}y^{i_{2}}\right)  \delta\in\mathbb{R}^{p+q+pq}\right]
\\
&  \in\left(  E\otimes\mathcal{W}_{D}\right)  _{\left(  x^{i},u^{j},u_{i_{1}%
}^{j}\right)  }\text{,}%
\end{align*}
for which we get

\begin{proposition}
\label{t4.2}We have
\[
(x^{i},u^{j};u_{i_{1}}^{j};u_{;i_{2}}^{j},u_{i_{1};i_{2}}^{j})\in
\hat{\mathbf{J}}^{2}(\pi)
\]
iff $u_{i}^{j}=u_{;i}^{j}$ for all $1\leq i\leq p$ and all $1\leq j\leq q$.
\end{proposition}

\begin{proof}
It is easy to see that $\left(  x^{i},u^{j},u_{i_{1}}^{j},u_{;i_{2}}%
^{j},u_{i_{1};i_{2}}^{j}\right)  \in\mathbf{J}^{1}(\pi_{1})$ is $\pi_{1,0}%
$-related to $\left(  x^{i},u^{j},u_{i_{1}}^{j}\right)  \in\mathbf{J}^{1}%
(\pi)$ iff
\[
u^{j}+\delta\Sigma_{i_{2}=1}^{n}y^{i_{2}}u_{;i_{2}}^{j}=u^{j}+\delta
\Sigma_{i_{1}=1}^{n}y^{i_{1}}u_{i_{1}}^{j}\qquad(1\leq j\leq q)
\]
for all $(y^{1},...,y^{p})\in\mathbb{R}^{p}$ and all $\delta\in\mathbb{R}$,
which is tantamount to saying that
\[
u_{;i}^{j}=u_{i}^{j}%
\]
for all $1\leq i\leq p$ and all $1\leq j\leq q$. This completes the proof.
\end{proof}

\begin{notation}
Thus the coordinate $(x^{i},u^{j};u_{i_{1}}^{j};u_{;i_{2}}^{j},u_{i_{1};i_{2}%
}^{j})\in\hat{\mathbf{J}}^{2}(\pi)$ can be simplified to $(x^{i}%
,u^{j},u_{i_{1}}^{j},u_{i_{1};i_{2}}^{j})$.
\end{notation}

Now we take a step forward.

\begin{proposition}
\label{t4.3}Let $(x^{i},u^{j},u_{i_{1}}^{j},u_{i_{1};i_{2}}^{j})\in
\hat{\mathbf{J}}^{2}(\pi)$. Then $(x^{i},u^{j},u_{i_{1}}^{j},u_{i_{1};i_{2}%
}^{j})\in\mathbf{J}^{2}(\pi)$ iff
\[
u_{i_{1};i_{2}}^{j}=u_{i_{2};i_{1}}^{j}%
\]
for all $1\leq i_{1},i_{2}\leq p$ and all $1\leq j\leq q$.
\end{proposition}

\begin{proof}
Let $\gamma\in\left(  M\otimes\mathcal{W}_{D^{2}}\right)  _{(x^{i})}$. Then
$\gamma$ is of the Taylor representation
\begin{align*}
(\delta_{1},\delta_{2})  &  \in\mathbb{R}^{2}\longmapsto(x^{1},...,x^{p}%
)+\delta_{1}(y_{1}^{1},...,y_{1}^{p})+\delta_{2}(y_{2}^{1},...,y_{2}%
^{p})+\delta_{1}\delta_{2}(y_{12}^{1},...,y_{12}^{p})\\
&  =(x^{1}+\delta_{1}y_{1}^{1},...,x^{p}+\delta_{1}y_{1}^{p})+\delta_{2}%
(y_{2}^{1}+\delta_{1}y_{12}^{1},...,y_{2}^{p}+\delta_{1}y_{12}^{p})\\
&  =(x^{1}+\delta_{2}y_{2}^{1},...,x^{p}+\delta_{2}y_{2}^{p})+\delta_{1}%
(y_{1}^{1}+\delta_{2}y_{12}^{1},...,y_{1}^{p}+\delta_{2}y_{12}^{p})\\
&  \in\mathbb{R}^{p}%
\end{align*}
Let $\nabla_{(x^{i},u^{j},u_{i_{1}}^{j})}=\left(  x^{i},u^{j},u_{i_{1}}%
^{j},u_{i_{1};i_{2}}^{j}\right)  $ and $\nabla_{(x^{i},u^{j})}=(x^{i}%
,u^{j},u_{i_{1}}^{j})$. Then we have
\begin{align*}
&  \delta_{1}\in\mathbb{R\mapsto}\nabla_{(x^{i},u^{j})}(\gamma\left(
\cdot,0\right)  )\left(  \delta_{1}\right) \\
&  =\delta_{1}\in\mathbb{R\mapsto}(x^{i}+\delta_{1}y_{1}^{i},u^{j}+\delta
_{1}\Sigma_{i_{1}=1}^{p}y_{1}^{i_{1}}u_{i_{1}}^{j})\in\mathbb{R}^{p+q}\\
&  \delta_{1}\in\mathbb{R\mapsto}\nabla_{(x^{i},u^{j},u_{i_{1}}^{j})}%
(\gamma\left(  \cdot,0\right)  )\left(  \delta_{1}\right)  \in\mathbb{R}%
^{p+q+pq}\\
&  =\delta_{1}\in\mathbb{R\mapsto}(x^{i}+\delta_{1}y_{1}^{i},u^{j}+\delta
_{1}\sum_{i_{1}=1}^{p}y_{1}^{i_{1}}u_{i_{1}}^{j},u_{i_{1}}^{j}+\delta_{1}%
\sum_{i_{2}=1}^{p}y_{1}^{i_{2}}u_{i_{1};i_{2}}^{j})\in\mathbb{R}^{p+q+pq}%
\end{align*}
while we have
\begin{align}
&  (\delta_{1},\delta_{2})\in\mathbb{R}^{2}\longmapsto\nabla_{\nabla
_{(x^{i},u^{j})}(\gamma(\cdot,0))(\delta_{1})}(\gamma(\delta_{1}%
,\cdot))(\delta_{2})\in\mathbb{R}^{p+q}\nonumber\\
&  =(\delta_{1},\delta_{2})\in\mathbb{R}^{2}\longmapsto\left(
\begin{array}
[c]{c}%
x^{i}+\delta_{1}y_{1}^{i}+\delta_{2}y_{2}^{i}+\delta_{1}\delta_{2}y_{12}%
^{i},u^{j}+\delta_{1}\sum_{i_{1}=1}^{p}y_{1}^{i_{1}}u_{i_{1}}^{j}+\\
\delta_{2}\sum_{i_{1}=1}^{p}\left(  y_{2}^{i_{1}}+\delta_{1}y_{12}^{i_{1}%
}\right)  \left(  u_{i_{1}}^{p}+\delta_{1}\sum_{i_{2}=1}^{p}y_{1}^{i_{2}%
}u_{i_{1};i_{2}}^{j}\right)
\end{array}
\right)  \in\mathbb{R}^{p+q}\nonumber\\
&  =(\delta_{1},\delta_{2})\in\mathbb{R}^{2}\longmapsto\left(
\begin{array}
[c]{c}%
x^{i}+\delta_{1}y_{1}^{i}+\delta_{2}y_{2}^{i}+\delta_{1}\delta_{2}y_{12}%
^{i},u^{j}+\\
\delta_{1}\sum_{i_{1}=1}^{p}y_{1}^{i_{1}}u_{i_{1}}^{j}+\delta_{2}\sum
_{i_{1}=1}^{p}y_{2}^{i_{1}}u_{i_{1}}^{j}+\\
\delta_{1}\delta_{2}\sum_{i_{1}=1}^{p}y_{12}^{i_{1}}u_{i_{1}}^{j}+\delta
_{1}\delta_{2}\sum_{i_{2}=1}^{p}\sum_{i_{1}=1}^{p}y_{1}^{i_{2}}y_{2}^{i_{1}%
}u_{i_{1};i_{2}}^{j}%
\end{array}
\right)  \in\mathbb{R}^{p+q} \label{4.3.1}%
\end{align}
On the other hand, we have
\begin{align*}
\delta_{2}  &  \in\mathbb{R\mapsto}\nabla_{(x^{i},u^{j})}(\gamma
(0,\cdot))(\delta_{2})\in\mathbb{R}^{p+q}\\
&  =\delta_{2}\in\mathbb{R\mapsto}\left(  x^{i}+\delta_{2}y_{2}^{i}%
,u^{j}+\delta_{2}\sum_{i_{1}=1}^{p}y_{2}^{i_{1}}u_{i_{1}}^{j}\right)
\in\mathbb{R}^{p+q}\\
\delta_{2}  &  \in\mathbb{R\mapsto}\nabla_{(x^{i},u^{j},u_{i_{1}}^{j})}%
(\gamma(0,\cdot))(\delta_{2})\in\mathbb{R}^{p+q+pq}\\
&  =\delta_{2}\in\mathbb{R\mapsto}\left(  x^{i}+\delta_{2}y_{2}^{i}%
,u^{j}+\delta_{2}\sum_{i_{1}=1}^{p}y_{2}^{i_{1}}u_{i_{1}}^{j},u_{i_{1}}%
^{j}+\delta_{2}\sum_{i_{2}=1}^{p}y_{2}^{i_{2}}u_{i_{1};i_{2}}^{j}\right)
\in\mathbb{R}^{p+q+pq}%
\end{align*}
while we have
\begin{align}
&  (\delta_{1},\delta_{2})\in\mathbb{R}^{2}\longmapsto\nabla_{\nabla
_{(x^{i},u^{j})}(\gamma(0,\cdot))(\delta_{2})}(\gamma(\cdot,\delta
_{2}))(\delta_{1})\in\mathbb{R}^{p+q}\nonumber\\
&  =(\delta_{1},\delta_{2})\in\mathbb{R}^{2}\longmapsto\left(
\begin{array}
[c]{c}%
x^{i}+\delta_{1}y_{1}^{i}+\delta_{2}y_{2}^{i}+\delta_{1}\delta_{2}y_{3}%
^{i},u^{j}+\delta_{2}\sum_{i_{1}=1}^{p}y_{2}^{i_{1}}u_{i_{1}}^{j}+\\
\delta_{1}\sum_{i_{1}=1}^{p}\left(  y_{1}^{i_{1}}+\delta_{2}y_{3}^{i_{1}%
}\right)  \left(  u_{i_{1}}^{j}+\delta_{2}\sum_{i_{2}=1}^{p}y_{2}^{i_{2}%
}u_{i_{1};i_{2}}^{j}\right)
\end{array}
\right)  \in\mathbb{R}^{p+q}\nonumber\\
&  =(\delta_{1},\delta_{2})\in\mathbb{R}^{2}\longmapsto\left(
\begin{array}
[c]{c}%
x^{i}+\delta_{1}y_{1}^{i}+\delta_{2}y_{2}^{i}+\delta_{1}\delta_{2}y_{3}%
^{i},u^{j}+\\
\delta_{1}\sum_{i_{1}=1}^{p}y_{1}^{i_{1}}u_{i_{1}}^{j}+\delta_{2}\sum
_{i_{1}=1}^{p}y_{2}^{i_{1}}u_{i_{1}}^{j}+\\
\delta_{1}\delta_{2}\sum_{i_{1}=1}^{p}y_{3}^{i_{1}}u_{i_{1}}^{j}+\delta
_{1}\delta_{2}\sum_{i_{2}=1}^{p}\sum_{i_{1}=1}^{p}y_{1}^{i_{1}}y_{2}^{i_{2}%
}u_{i_{1};i_{2}}^{j}%
\end{array}
\right)  \in\mathbb{R}^{p+q}\nonumber\\
&  =(\delta_{1},\delta_{2})\in\mathbb{R}^{2}\longmapsto\left(
\begin{array}
[c]{c}%
x^{i}+\delta_{1}y_{1}^{i}+\delta_{2}y_{2}^{i}+\delta_{1}\delta_{2}y_{3}%
^{i},u^{j}+\\
\delta_{1}\sum_{i_{1}=1}^{p}y_{1}^{i_{1}}u_{i_{1}}^{j}+\delta_{2}\sum
_{i_{1}=1}^{p}y_{2}^{i_{1}}u_{i_{1}}^{j}+\\
\delta_{1}\delta_{2}\sum_{i_{2}=1}^{p}y_{3}^{i_{2}}u_{i_{2}}^{j}+\delta
_{1}\delta_{2}\sum_{i_{1}=1}^{p}\sum_{i_{2}=1}^{p}y_{1}^{i_{1}}y_{2}^{i_{2}%
}u_{i_{1};i_{2}}^{j}%
\end{array}
\right)  \in\mathbb{R}^{p+q} \label{4.3.2}%
\end{align}
Therefore it follows from (\ref{4.3.1}) and (\ref{4.3.2}) that
\begin{align*}
(\delta_{1},\delta_{2})  &  \in\mathbb{R}^{2}\longmapsto\nabla_{\nabla
_{(x^{i},u^{j})}(\gamma(\cdot,0))(\delta_{1})}(\gamma(\delta_{1}%
,\cdot))(\delta_{2})\in\mathbb{R}^{p+q}\\
&  =(\delta_{1},\delta_{2})\in\mathbb{R}^{2}\longmapsto\nabla_{\nabla
_{(x^{i},u^{j})}(\gamma(0,\cdot))(\delta_{2})}(\gamma(\cdot,\delta
_{2}))(\delta_{1})\in\mathbb{R}^{p+q}%
\end{align*}
for all $\gamma\in\left(  M\otimes\mathcal{W}_{D^{2}}\right)  _{(x^{i})}$ iff
\[
u_{i_{1};i_{2}}^{j}=u_{i_{2};i_{1}}^{j}%
\]
for all $1\leq i_{1},i_{2}\leq p$ and all $1\leq j\leq q$. This completes the proof.
\end{proof}

\begin{definition}
Thus we have defined a bijection $\theta_{\mathbf{J}^{2}(\pi)}^{\mathcal{J}%
^{2}(\pi)}:\mathcal{J}^{2}(\pi)\rightarrow\mathbf{J}^{2}(\pi)$, which goes
formally as follows:
\begin{align*}
&  \theta_{\mathbf{J}^{2}(\pi)}^{\mathcal{J}^{2}(\pi)}\left(  x^{i}%
,u^{j},u_{i_{1}}^{j},u_{i_{1},i_{2}}^{j}\right) \\
&  =\left[  \delta\in\mathbb{R}\mapsto\left(  x^{i}\right)  +\left(
y^{i}\right)  \delta\in\mathbb{R}^{p}\right]  \in\left(  M\otimes
\mathcal{W}_{D}\right)  _{\left(  x^{i}\right)  }\mapsto\\
&  \left[  \delta\in\mathbb{R}\mapsto\theta_{\mathbf{J}^{1}(\pi)}%
^{\mathcal{J}^{1}(\pi)}\left(  \left(  x^{i},u^{j},u_{i_{1}}^{j}\right)
+\left(  y^{i},\sum_{i_{1}=1}^{p}u_{i_{1}}^{j}y^{i_{1}},\sum_{i_{2}=1}%
^{p}u_{i_{1},i_{2}}^{j}y^{i_{2}}\right)  \delta\right)  \in\mathbf{J}^{1}%
(\pi)\right] \\
&  \in\left(  \mathbf{J}^{1}(\pi)\otimes\mathcal{W}_{D}\right)  _{\theta
_{\mathbf{J}^{1}(\pi)}^{\mathcal{J}^{1}(\pi)}\left(  x^{i},u^{j},u_{i_{1}}%
^{j}\right)  }%
\end{align*}

\end{definition}

We can go on by induction on $n$.

\begin{theorem}
\label{t4.4}The mapping $\theta_{\mathbf{J}^{n+1}(\pi)}^{\mathcal{J}^{n+1}%
(\pi)}:\mathcal{J}^{n+1}(\pi)\rightarrow\mathbf{J}^{n+1}(\pi)$, which is
defined to be
\begin{align*}
&  \theta_{\mathbf{J}^{n+1}(\pi)}^{\mathcal{J}^{n+1}(\pi)}\left(  x^{i}%
,u^{j},u_{i_{1}}^{j},u_{i_{1},i_{2}}^{j},...,u_{i_{1},i_{2},...,i_{n+1}}%
^{j}\right) \\
&  =\left[  \delta\in\mathbb{R}\mapsto\left(  x^{i}\right)  +\left(
y^{i}\right)  \delta\in\mathbb{R}^{p}\right]  \in\left(  M\otimes
\mathcal{W}_{D}\right)  _{\left(  x^{i}\right)  }\mapsto\\
&  \left[
\begin{array}
[c]{c}%
\delta\in\mathbb{R}\mapsto\\
\theta_{\mathbf{J}^{n}(\pi)}^{\mathcal{J}^{n}(\pi)}\left(
\begin{array}
[c]{c}%
\left(  x^{i},u^{j},u_{i_{1}}^{j},...,u_{i_{1},i_{2},...,i_{n}}^{j}\right)
+\\
\left(  y^{i},\sum_{i_{1}=1}^{p}u_{i_{1}}^{j}y^{i_{1}},...,\sum_{i_{n+1}%
=1}^{p}u_{i_{1},i_{2},...,i_{n+1}}^{j}y^{i_{n+1}}\right)  \delta
\end{array}
\right)  \in\mathbf{J}^{n}(\pi)
\end{array}
\right] \\
&  \in\left(  \mathbf{J}^{n}(\pi)\otimes\mathcal{W}_{D}\right)  _{\theta
_{\mathbf{J}^{n}(\pi)}^{\mathcal{J}^{n}(\pi)}\left(  x^{i},u^{j},u_{i_{1}}%
^{j},...,u_{i_{1},i_{2},...,i_{n}}^{j}\right)  }%
\end{align*}
by induction on $n$, is bijective.
\end{theorem}

\section{\label{s5}The Second Approach with Coordinates}

\begin{definition}
We define mappings $\theta_{\mathbb{J}^{D^{n}}(\pi)}^{\mathcal{J}^{n}(\pi
)}:\mathcal{J}^{n}(\pi)\rightarrow\mathbb{J}^{D^{n}}(\pi)$ as $\varphi
_{n}\circ\theta_{\mathbf{J}^{n}(\pi)}^{\mathcal{J}^{n}(\pi)}$.
\end{definition}

\begin{remark}
Since $\mathbb{J}^{D}(\pi)=\mathbf{J}^{1}(\pi)$ and $\varphi_{1}$ is the
identity transformation, we have
\begin{align*}
&  \theta_{\mathbb{J}^{D^{1}}(\pi)}^{\mathcal{J}^{1}(\pi)}\left(  x^{i}%
,u^{j},u_{i_{1}}^{j}\right) \\
&  =\left[  \delta\in\mathbb{R}\longmapsto(x^{i})+\delta(y^{i})\in
\mathbb{R}^{p}\right]  \in\left(  M\otimes\mathcal{W}_{D}\right)  _{\left(
x^{i}\right)  }\mapsto\\
&  \left[  \delta\in\mathbb{R}\longmapsto(x^{i},u^{j})+\delta(y^{i}%
,\sum_{i_{1}=1}^{p}y^{i_{1}}u_{i_{1}}^{j})\in\mathbb{R}^{p+q}\right] \\
&  \in\left(  E\otimes\mathcal{W}_{D}\right)  _{\left(  x^{i},u^{j}\right)  }%
\end{align*}

\end{remark}

With due regard to Theorem \ref{t4.4}, it is easy to see that

\begin{lemma}
\label{t5.1}Given $\left(  x^{i},u^{j},u_{i_{1}}^{j},u_{i_{1},i_{2}}%
^{j},...,u_{i_{1},i_{2},...,i_{n+1}}^{j}\right)  \in\mathcal{J}^{n+1}(\pi)$,
we have
\begin{align*}
&  \theta_{\mathbb{J}^{D^{n+1}}(\pi)}^{\mathcal{J}^{n+1}(\pi)}\left(
x^{i},u^{j},u_{i_{1}}^{j},u_{i_{1},i_{2}}^{j},...,u_{i_{1},i_{2},...,i_{n+1}%
}^{j}\right) \\
&  =\left[  (\delta_{1},...,\delta_{n+1})\in\mathbb{R}^{n+1}\longmapsto
(x^{i})+\sum_{r=1}^{n+1}\sum_{1\leq k_{1}<...<k_{r}\leq n+1}\delta_{k_{1}%
}...\delta_{k_{r}}(y_{k_{1},...,k_{r}}^{i})\in\mathbb{R}^{p}\right] \\
&  \in\left(  M\otimes\mathcal{W}_{D^{n+1}}\right)  _{\left(  x^{i}\right)
}\mapsto\\
&  \left[
\begin{array}
[c]{c}%
(\delta_{1},...,\delta_{n+1})\in\mathbb{R}^{n+1}\longmapsto\\
\theta_{\mathbb{J}^{D^{n}}(\pi)}^{\mathcal{J}^{n}(\pi)}\left(
\begin{array}
[c]{c}%
\left(  x^{i},u^{j},u_{i_{1}}^{j},...,u_{i_{1},i_{2},...,i_{n}}^{j}\right)
+\\
\left(  y_{n+1}^{i},\sum_{i_{1}=1}^{p}u_{i_{1}}^{j}y_{n+1}^{i_{1}}%
,...,\sum_{i_{n+1}=1}^{p}u_{i_{1},i_{2},...,i_{n+1}}^{j}y_{n+1}^{i_{n+1}%
}\right)  \delta_{n+1}%
\end{array}
\right) \\
\left(
\begin{array}
[c]{c}%
(\delta_{1},...,\delta_{n})\in\mathbb{R}^{n}\longmapsto(x^{i})+\\
\sum_{r=1}^{n}\sum_{1\leq k_{1}<...<k_{r}\leq n}\delta_{k_{1}}...\delta
_{k_{r}}(y_{k_{1},...,k_{r}}^{i}+y_{k_{1},...,k_{r},n+1}^{i}\delta_{n+1}%
)\in\mathbb{R}^{p}%
\end{array}
\right)
\end{array}
\right] \\
&  \in\left(  E\otimes\mathcal{W}_{D^{n+1}}\right)  _{\left(  x^{i}%
,u^{j}\right)  }%
\end{align*}

\end{lemma}

Now we are going to determine $\theta_{\mathbb{J}^{D^{2}}(\pi)}^{\mathcal{J}%
^{2}(\pi)}$.

\begin{theorem}
\label{t5.2}Given $\left(  x^{i},u^{j},u_{i_{1}}^{j},u_{i_{1,}i_{2}}%
^{j}\right)  \in\mathcal{J}^{2}(\pi)$, we have
\begin{align*}
&  \theta_{\mathbb{J}^{D^{2}}(\pi)}^{\mathcal{J}^{2}(\pi)}\left(  x^{i}%
,u^{j},u_{i_{1}}^{j},u_{i_{1,}i_{2}}^{j}\right) \\
&  =\left[  \left(  \delta_{1},\delta_{2}\right)  \in\mathbb{R}^{2}%
\mapsto\left(  x^{i}\right)  +\left(  y_{1}^{i}\right)  \delta_{1}+\left(
y_{2}^{i}\right)  \delta_{2}+\left(  y_{12}^{i}\right)  \delta_{1}\delta
_{2}\in\mathbb{R}^{p}\right] \\
&  \in\left(  M\otimes\mathcal{W}_{D^{2}}\right)  _{\left(  x^{i}\right)
}\mapsto\\
&  \left[
\begin{array}
[c]{c}%
\left(  \delta_{1},\delta_{2}\right)  \in\mathbb{R}^{2}\mapsto(x^{i}%
,u^{j})+(y_{1}^{i},\sum_{i_{1}=1}^{p}y_{1}^{i_{1}}u_{i_{1}}^{j})\delta
_{1}+(y_{2}^{i},\sum_{i_{1}=1}^{p}y_{2}^{i_{1}}u_{i_{1}}^{j})\delta_{2}+\\
(y_{12}^{i},\sum_{i_{1}=1}^{p}\sum_{i_{2}=1}^{p}y_{1}^{i_{1}}y_{2}^{i_{2}%
}u_{i_{1},i_{2}}^{j}+\sum_{i_{1}=1}^{p}y_{12}^{i_{1}}u_{i_{1}}^{j})\delta
_{1}\delta_{2}\in\mathbb{R}^{p+q}%
\end{array}
\right] \\
&  \in\left(  E\otimes\mathcal{W}_{D^{2}}\right)  _{\left(  x^{i}%
,u^{j}\right)  }%
\end{align*}

\end{theorem}

\begin{proof}
The Taylor representation of
\begin{align*}
&  \theta_{\mathbb{J}^{D^{2}}(\pi)}^{\mathcal{J}^{2}(\pi)}\left(  x^{i}%
,u^{j},u_{i_{1}}^{j},u_{i_{1,}i_{2}}^{j}\right) \\
&  \left(  \left(  \delta_{1},\delta_{2}\right)  \in\mathbb{R}^{2}%
\mapsto\left(  x^{i}\right)  +\left(  y_{1}^{i}\right)  \delta_{1}+\left(
y_{2}^{i}\right)  \delta_{2}+\left(  y_{12}^{i}\right)  \delta_{1}\delta
_{2}\in\mathbb{R}^{p}\right) \\
&  =\left(  \delta_{1},\delta_{2}\right)  \in\mathbb{R}^{2}\mapsto\\
&  \theta_{\mathbb{J}^{D}(\pi)}^{\mathcal{J}^{1}(\pi)}\left(  \left(
x^{i},u^{j},u_{i_{1}}^{j}\right)  +\left(  y_{2}^{i},\sum_{i_{1}=1}^{p}%
y_{2}^{i_{1}}u_{i_{1}}^{j},\sum_{i_{2}=1}^{p}y_{2}^{i_{2}}u_{i_{1},i_{2}}%
^{j}\right)  \delta_{2}\right) \\
&  \left(  \delta_{1}\in\mathbb{R\mapsto}\left(  x^{i}+y_{2}^{i}\delta
_{2}\right)  +\left(  y_{1}^{i}+y_{12}^{i}\delta_{2}\right)  \delta_{1}\right)
\\
&  \in\mathbb{R}^{p+q}%
\end{align*}
goes as follows:
\begin{align*}
&  \left(  x^{i}+y_{2}^{i}\delta_{2},u^{j}+\left(  \sum_{i_{1}=1}^{p}%
y_{2}^{i_{1}}u_{i_{1}}^{j}\right)  \delta_{2}\right)  +\\
&  \left(  y_{1}^{i}+y_{12}^{i}\delta_{2},\sum_{i_{1}=1}^{p}\left(  u_{i_{1}%
}^{j}+\left(  \sum_{i_{2}=1}^{p}y_{2}^{i_{2}}u_{i_{1},i_{2}}^{j}\right)
\delta_{2}\right)  \left(  y_{1}^{i_{1}}+y_{12}^{i_{1}}\delta_{2}\right)
\right)  \delta_{1}\\
&  =(x^{i},u^{j})+(y_{1}^{i},\sum_{i_{1}=1}^{p}y_{1}^{i_{1}}u_{i_{1}}%
^{j})\delta_{1}+(y_{2}^{i},\sum_{i_{1}=1}^{p}y_{2}^{i_{1}}u_{i_{1}}^{j}%
)\delta_{2}+\\
&  (y_{12}^{i},\sum_{i_{1}=1}^{p}\sum_{i_{2}=1}^{p}y_{1}^{i_{1}}y_{2}^{i_{2}%
}u_{i_{1},i_{2}}^{j}+\sum_{i_{1}=1}^{p}y_{12}^{i_{1}}u_{i_{1}}^{j})\delta
_{1}\delta_{2}%
\end{align*}
so that we have the coveted result.
\end{proof}

Generally, by the same token, we have

\begin{theorem}
\label{t5.3}Given $(x^{i},u^{j},u_{i_{1}}^{j},u_{i_{1},i_{2}}^{j}%
,...,u_{i_{1},i_{2},...,i_{n}}^{j})\in\mathcal{J}^{n}(\pi)$, we have
\begin{align*}
&  \theta_{\mathbb{J}^{D^{n}}(\pi)}^{\mathcal{J}^{n}(\pi)}(x^{i}%
,u^{j},u_{i_{1}}^{j},u_{i_{1},i_{2}}^{j},...,u_{i_{1},i_{2},...,i_{n}}^{j})\\
&  =\left[  (\delta_{1},...,\delta_{n})\in\mathbb{R}^{n}\longmapsto
(x^{i})+\sum_{r=1}^{n}\sum_{1\leq k_{1}<...<k_{r}\leq n}\delta_{k_{1}%
}...\delta_{k_{r}}(y_{k_{1},...,k_{r}}^{i})\in\mathbb{R}^{p}\right] \\
&  \in\left(  M\otimes\mathcal{W}_{D^{n}}\right)  _{\left(  x^{i}\right)
}\mapsto\\
&  \left[
\begin{array}
[c]{c}%
(\delta_{1},...,\delta_{n})\in\mathbb{R}^{n}\longmapsto(x^{i},u^{j})+\\
\sum_{r=1}^{n}\sum_{1\leq k_{1}<...<k_{r}\leq n}\delta_{k_{1}}...\delta
_{k_{r}}(y_{k_{1},...,k_{r}}^{i},\sum\sum_{i_{1}=1}^{p}...\sum_{i_{s}=1}%
^{p}y_{\mathbf{J}_{1}}^{i_{1}}...y_{\mathbf{J}_{s}}^{i_{s}}u_{i_{1},...,i_{s}%
}^{j})\in\mathbb{R}^{p+q}%
\end{array}
\right] \\
&  \in\left(  E\otimes\mathcal{W}_{D^{n}}\right)  _{\left(  x^{i}%
,u^{j}\right)  }%
\end{align*}
where the completely undecorated $\sum\ $is taken over all partitions of the
set$\ \{k_{1},...,k_{r}\}\ $into nonempty subsets $\{\mathbf{J}_{1}%
,...,\mathbf{J}_{s}\}$, and if$\ \mathbf{J}=\{k_{1},...,k_{t}\}\ $is a set of
natural numbers with $k_{1}<...<k_{t}$, then $y_{\mathbf{J}}^{i}\ $denotes
$y_{k_{1},\cdots,k_{t}}^{i}$.
\end{theorem}

\begin{proof}
By using Lemma \ref{t5.1}, we can proceed by induction on $n$. The details are
safely left to the reader.
\end{proof}

\begin{definition}
We define mappings $\theta_{\mathbb{S}^{D^{n}}(\pi)}^{\mathcal{S}^{n}(\pi
)}:\mathcal{S}^{n}(\pi)\rightarrow\mathbb{S}^{D^{n}}(\pi)$ to be
\begin{align*}
&  \theta_{\mathbb{S}^{D^{n}}(\pi)}^{\mathcal{S}^{n}(\pi)}\left(  x^{i}%
,u^{j},u_{i_{1},i_{2},...,i_{n}}^{j}\right) \\
&  =\left[  (\delta_{1},...,\delta_{n})\in\mathbb{R}^{n}\longmapsto
(x^{i})+\sum_{r=1}^{n}\sum_{1\leq k_{1}<...<k_{r}\leq n}\delta_{k_{1}%
}...\delta_{k_{r}}(y_{k_{1},...,k_{r}}^{i})\in\mathbb{R}^{p}\right] \\
&  \in\left(  M\otimes\mathcal{W}_{D^{n}}\right)  _{\left(  x^{i}\right)
}\mapsto\\
&  \left[  \delta\in\mathbb{R}\longmapsto\left(  x^{i},u^{j}+\delta\sum_{1\leq
i_{1}\leq...\leq i_{n}\leq p}y_{1}^{i_{1}}...y_{n}^{i_{n}}u_{i_{1}%
,i_{2},...,i_{n}}^{j}\right)  \in\mathbb{R}^{p+q}\right] \\
&  \in\left(  E\otimes\mathcal{W}_{D}\right)  _{\left(  x^{i},u^{j}\right)
}^{\perp}%
\end{align*}

\end{definition}

It is easy to see that

\begin{proposition}
\label{t5.4}The mappings $\theta_{\mathbb{S}^{D^{n}}(\pi)}^{\mathcal{S}%
^{n}(\pi)}:\mathcal{S}^{n}(\pi)\rightarrow\mathbb{S}^{D^{n}}(\pi)$\ are bijective.
\end{proposition}

It is also easy to see that

\begin{proposition}
\label{t5.5}Given $\nabla=(x^{i},u^{j},u_{i_{1}}^{j},u_{i_{1},i_{2}}%
^{j},...,u_{i_{1},i_{2},...,i_{n+1}}^{j}),\nabla^{\prime}=(y^{i}%
,v^{j},v_{i_{1}}^{j},v_{i_{1},i_{2}}^{j},...,v_{i_{1},i_{2},...,i_{n+1}}%
^{j})\mathcal{J}^{n+1}(\pi)$, we have
\[
\pi_{n}^{n+1}\left(  \nabla\right)  =\pi_{n}^{n+1}\left(  \nabla^{\prime
}\right)
\]
iff
\[
\pi_{n}^{n+1}\left(  \theta_{\mathbb{J}^{D^{n+1}}(\pi)}^{\mathcal{J}^{n+1}%
(\pi)}\left(  \nabla\right)  \right)  =\pi_{n}^{n+1}\left(  \theta
_{\mathbb{J}^{D^{n+1}}(\pi)}^{\mathcal{J}^{n+1}(\pi)}\left(  \nabla^{\prime
}\right)  \right)  \text{,}%
\]
in which we get
\[
\theta_{\mathbb{S}^{D^{n+1}}(\pi)}^{\mathcal{S}^{n+1}(\pi)}\left(
\nabla\overset{\cdot}{-}\nabla^{\prime}\right)  =\theta_{\mathbb{J}^{D^{n+1}%
}(\pi)}^{\mathcal{J}^{n+1}(\pi)}\left(  \nabla\right)  \overset{\cdot}%
{-}\theta_{\mathbb{J}^{D^{n+1}}(\pi)}^{\mathcal{J}^{n+1}(\pi)}\left(
\nabla^{\prime}\right)
\]

\end{proposition}

\begin{theorem}
\label{t5.6}The mappings $\theta_{\mathbb{J}^{D^{n}}(\pi)}^{\mathcal{J}%
^{n}(\pi)}:\mathcal{J}^{n}(\pi)\rightarrow\mathbb{J}^{D^{n}}(\pi)$\ are bijective.
\end{theorem}

\begin{proof}
We proceed by induction on $n$. The mapping $\theta_{\mathbb{J}^{D}(\pi
)}^{\mathcal{J}^{1}(\pi)}:\mathcal{J}^{1}(\pi)\rightarrow\mathbb{J}^{D}(\pi
)$\ is obviously bijective. By Proposition \ref{t5.5} and the induction
hypothesis, $\left(  \theta_{\mathbb{J}^{D^{n+1}}(\pi)}^{\mathcal{J}^{n+1}%
(\pi)},\theta_{\mathbb{S}^{D^{n+1}}(\pi)}^{\mathcal{S}^{n+1}(\pi)}\underset
{E}{\times}\theta_{\mathbb{J}^{D^{n}}(\pi)}^{\mathcal{J}^{n}(\pi)}%
,\theta_{\mathbb{J}^{D^{n}}(\pi)}^{\mathcal{J}^{n}(\pi)}\right)  $\ gives a
morphism of affine bundles from the affine bundle $\pi_{n+1,n}:\mathcal{J}%
^{n+1}(\pi)\rightarrow\mathcal{J}^{n}(\pi)$ over the vector bundle
$\mathcal{S}^{n+1}(\pi)\underset{E}{\times}\mathcal{J}^{n}(\pi)\rightarrow
\mathcal{J}^{n}(\pi)$\ to the affine bundle $\pi_{n+1,n}:\mathbb{J}^{D^{n+1}%
}(\pi)\rightarrow\mathbb{J}^{D^{n}}(\pi)$\ over the vector bundle
$\mathbb{S}^{D^{n+1}}(\pi)\underset{E}{\times}\mathbb{J}^{D^{n}}%
(\pi)\rightarrow\mathbb{J}^{D^{n}}(\pi)$\ is an isomorphism of affine bundles,
so that the mapping $\theta_{\mathbb{J}^{D^{n+1}}(\pi)}^{\mathcal{J}^{n+1}%
(\pi)}:\mathcal{J}^{n+1}(\pi)\rightarrow\mathbb{J}^{D^{n+1}}(\pi)$\ is bijective.
\end{proof}

\begin{corollary}
The mappings $\varphi_{n}:\mathbf{J}^{n}\left(  \pi\right)  \rightarrow
\mathbb{J}^{D^{n}}(\pi)$\ are bijective.
\end{corollary}

\begin{proof}
This follows simply from Theorems \ref{t4.4}\ and \ref{t5.6}\ and the
commutativity of the following diagram:
\[%
\begin{array}
[c]{ccc}
& \mathcal{J}^{n}(\pi) & \\
\theta_{\mathbf{J}^{n}\left(  \pi\right)  }^{\mathcal{J}^{n}(\pi)}\swarrow &
& \searrow\theta_{\mathbb{J}^{D^{n}}(\pi)}^{\mathcal{J}^{n}(\pi)}\\
\mathbf{J}^{n}\left(  \pi\right)  &
\begin{array}
[c]{c}%
\rightarrow\\
\varphi_{n}%
\end{array}
& \mathbb{J}^{D^{n}}(\pi)
\end{array}
\]

\end{proof}

\section{\label{s6}The Third Approach with Coordinates}

\begin{definition}
We define mappings $\theta_{\mathbb{J}^{D_{n}}(\pi)}^{\mathcal{J}^{n}(\pi
)}:\mathcal{J}^{n}(\pi)\rightarrow\mathbb{J}_{D_{n}}(\pi)$ as $\psi_{n}%
\circ\theta_{\mathbb{J}^{D^{n}}(\pi)}^{\mathcal{J}^{n}(\pi)}$.
\end{definition}

Now we are going to determine $\theta_{\mathbb{J}^{D_{2}}(\pi)}^{\mathcal{J}%
^{2}(\pi)}$.

\begin{theorem}
\label{t6.2}Given $\left(  x^{i},u^{j},u_{i_{1}}^{j},u_{i_{1,}i_{2}}%
^{j}\right)  \in\mathcal{J}^{2}(\pi)$, we have
\begin{align*}
&  \theta_{\mathbb{J}^{D_{2}}(\pi)}^{\mathcal{J}^{2}(\pi)}\left(  x^{i}%
,u^{j},u_{i_{1}}^{j},u_{i_{1,}i_{2}}^{j}\right) \\
&  =\left[  \delta\in\mathbb{R}\longmapsto\left(  x^{i}\right)  +\left(
y_{1}^{i}\right)  \delta+\frac{1}{2}\left(  y_{2}^{i}\right)  \delta^{2}%
\in\mathbb{R}^{p}\right] \\
&  \in\left(  M\otimes\mathcal{W}_{D_{2}}\right)  _{\left(  x^{i}\right)
}\mapsto\\
&  \left[
\begin{array}
[c]{c}%
\delta\in\mathbb{R}\longmapsto(x^{i},u^{j})+(y_{1}^{i},\sum_{i_{1}=1}^{p}%
y_{1}^{i_{1}}u_{i_{1}}^{j})\delta+\\
\frac{1}{2}(y_{2}^{i},\sum_{i_{1}=1}^{p}\sum_{i_{2}=1}^{p}y_{1}^{i_{1}}%
y_{1}^{i_{2}}u_{i_{1},i_{2}}^{j}+\sum_{i_{1}=1}^{p}y_{2}^{i_{1}}u_{i_{1}}%
^{j})\delta^{2}\in\mathbb{R}^{p+q}%
\end{array}
\right] \\
&  \in\left(  E\otimes\mathcal{W}_{D_{2}}\right)  _{\left(  x^{i}%
,u^{j}\right)  }%
\end{align*}

\end{theorem}

\begin{proof}
The Taylor representation of
\[
\left(  \mathrm{id}_{M}\otimes\mathcal{W}_{\left(  d_{1},d_{2}\right)  \in
D^{2}\mapsto d_{1}+d_{2}\in D_{2}}\right)  \left(  \delta\in\mathbb{R}%
\longmapsto\left(  x^{i}\right)  +\left(  y_{1}^{i}\right)  \delta+\frac{1}%
{2}\left(  y_{2}^{i}\right)  \delta^{2}\in\mathbb{R}^{p}\right)
\]
is
\begin{align*}
&  \left(  \delta_{1},\delta_{2}\right)  \in\mathbb{R}^{2}\longmapsto\left(
x^{i}\right)  +\left(  y_{1}^{i}\right)  \left(  \delta_{1}+\delta_{2}\right)
+\frac{1}{2}\left(  y_{2}^{i}\right)  \left(  \delta_{1}+\delta_{2}\right)
^{2}\in\mathbb{R}^{p}\\
&  =\left(  \delta_{1},\delta_{2}\right)  \in\mathbb{R}^{2}\longmapsto\left(
x^{i}\right)  +\left(  y_{1}^{i}\right)  \delta_{1}+\left(  y_{1}^{i}\right)
\delta_{2}+\left(  y_{2}^{i}\right)  \delta_{1}\delta_{2}\in\mathbb{R}^{p}%
\end{align*}
so that its transformation under the mapping $\theta_{\mathbb{J}^{D^{2}}(\pi
)}^{\mathcal{J}^{2}(\pi)}\left(  x^{i},u^{j},u_{i_{1}}^{j},u_{i_{1,}i_{2}}%
^{j}\right)  $ is
\begin{align*}
\left(  \delta_{1},\delta_{2}\right)   &  \in\mathbb{R}^{2}\mapsto(x^{i}%
,u^{j})+(y_{1}^{i},\sum_{i_{1}=1}^{p}y_{1}^{i_{1}}u_{i_{1}}^{j})\delta
_{1}+(y_{1}^{i},\sum_{i_{1}=1}^{p}y_{1}^{i_{1}}u_{i_{1}}^{j})\delta_{2}+\\
&  (y_{2}^{i},\sum_{i_{1}=1}^{p}\sum_{i_{2}=1}^{p}y_{1}^{i_{1}}y_{1}^{i_{2}%
}u_{i_{1},i_{2}}^{j}+\sum_{i_{1}=1}^{p}y_{2}^{i_{1}}u_{i_{1}}^{j})\delta
_{1}\delta_{2}\\
&  \in\mathbb{R}^{p+q}\\
&  =\left(  \delta_{1},\delta_{2}\right)  \in\mathbb{R}^{2}\mapsto(x^{i}%
,u^{j})+(y_{1}^{i},\sum_{i_{1}=1}^{p}y_{1}^{i_{1}}u_{i_{1}}^{j})\left(
\delta_{1}+\delta_{2}\right)  +\\
&  \frac{1}{2}(y_{2}^{i},\sum_{i_{1}=1}^{p}\sum_{i_{2}=1}^{p}y_{1}^{i_{1}%
}y_{1}^{i_{2}}u_{i_{1},i_{2}}^{j}+\sum_{i_{1}=1}^{p}y_{2}^{i_{1}}u_{i_{1}}%
^{j})\left(  \delta_{1}+\delta_{2}\right)  ^{2}\\
&  \in\mathbb{R}^{p+q}%
\end{align*}
Therefore we have the coveted result.
\end{proof}

Generally, by the same token, we have

\begin{theorem}
\label{t6.3}Given $(x^{i},u^{j},u_{i_{1}}^{j},u_{i_{1},i_{2}}^{j}%
,...,u_{i_{1},i_{2},...,i_{n}}^{j})\in\mathcal{J}^{n}(\pi)$, we have
\begin{align*}
&  \theta_{\mathbb{J}^{D_{n}}(\pi)}^{\mathcal{J}^{n}(\pi)}(x^{i}%
,u^{j},u_{i_{1}}^{j},u_{i_{1},i_{2}}^{j},...,u_{i_{1},i_{2},...,i_{n}}^{j})\\
&  =\left[  \delta\in\mathbb{R}\longmapsto(x^{i})+\sum_{k=1}^{n}\frac
{\delta^{k}}{k!}(y_{k}^{i})\in\mathbb{R}^{p}\right]  \in\left(  M\otimes
\mathcal{W}_{D_{n}}\right)  _{\left(  x^{i}\right)  }\mapsto\\
&  \left[  \delta\in\mathbb{R}\longmapsto(x^{i},u^{j})+\sum_{k=1}^{n}%
\frac{\delta^{k}}{k!}\sum\sum_{i_{1}=1}^{p}...\sum_{i_{r}=1}^{p}\left(
y_{k}^{i},u_{i_{1},...,i_{r}}^{j}y_{k_{1}}^{i_{1}}...y_{k_{r}}^{i_{r}}\right)
\in\mathbb{R}^{p+q}\right] \\
&  \in\left(  E\otimes\mathcal{W}_{D_{n}}\right)  _{\left(  x^{i}%
,u^{j}\right)  }%
\end{align*}
where the undecorated $\sum\ $is taken over all partitions of the positive
integer $k\ $into positive integers $k_{1},...,k_{r}$ (so that $k=k_{1}%
+...+k_{r}$) with $1\leq k_{1}\leq...\leq k_{r}\leq n$.
\end{theorem}

\begin{definition}
We define mappings $\theta_{\mathbb{S}^{D_{n}}(\pi)}^{\mathcal{S}^{n}(\pi
)}:\mathcal{S}^{n}(\pi)\rightarrow\mathbb{S}^{D_{n}}(\pi)$ to be
\begin{align*}
&  \theta_{\mathbb{S}^{D_{n}}(\pi)}^{\mathcal{S}^{n}(\pi)}\left(  x^{i}%
,u^{j},u_{i_{1},i_{2},...,i_{n}}^{j}\right) \\
&  =\left[  \delta\in\mathbb{R}\longmapsto(x^{i})+\sum_{k=1}^{n}\frac
{\delta^{k}}{k!}(y_{k}^{i})\in\mathbb{R}^{p}\right]  \in\left(  M\otimes
\mathcal{W}_{D_{n}}\right)  _{\left(  x^{i}\right)  }\mapsto\\
&  \left[  \delta\in\mathbb{R}\longmapsto\left(  x^{i},u^{j}+\frac{\delta}%
{n!}\sum_{1\leq i_{1}\leq...\leq i_{n}\leq p}y_{1}^{i_{1}}...y_{1}^{i_{n}%
}u_{i_{1},i_{2},...,i_{n}}^{j}\right)  \in\mathbb{R}^{p+q}\right] \\
&  \in\left(  E\otimes\mathcal{W}_{D}\right)  _{\left(  x^{i},u^{j}\right)
}^{\perp}%
\end{align*}

\end{definition}

It is easy to see that

\begin{proposition}
\label{t6.4}The mappings $\theta_{\mathbb{S}^{D_{n}}(\pi)}^{\mathcal{S}%
^{n}(\pi)}:\mathcal{S}^{n}(\pi)\rightarrow\mathbb{S}^{D_{n}}(\pi)$\ are bijective.
\end{proposition}

It is also easy to see that

\begin{proposition}
\label{t6.5}Given $\nabla=(x^{i},u^{j},u_{i_{1}}^{j},u_{i_{1},i_{2}}%
^{j},...,u_{i_{1},i_{2},...,i_{n+1}}^{j}),\nabla^{\prime}=(y^{i}%
,v^{j},v_{i_{1}}^{j},v_{i_{1},i_{2}}^{j},...,v_{i_{1},i_{2},...,i_{n+1}}%
^{j})\mathcal{J}^{n+1}(\pi)$, we have
\[
\pi_{n}^{n+1}\left(  \nabla\right)  =\pi_{n}^{n+1}\left(  \nabla^{\prime
}\right)
\]
iff
\[
\pi_{n}^{n+1}\left(  \theta_{\mathbb{J}^{D_{n+1}}(\pi)}^{\mathcal{J}^{n+1}%
(\pi)}\left(  \nabla\right)  \right)  =\pi_{n}^{n+1}\left(  \theta
_{\mathbb{J}^{D_{n+1}}(\pi)}^{\mathcal{J}^{n+1}(\pi)}\left(  \nabla^{\prime
}\right)  \right)  \text{,}%
\]
in which we get
\[
\theta_{\mathbb{S}^{D_{n+1}}(\pi)}^{\mathcal{S}^{n+1}(\pi)}\left(
\nabla\overset{\cdot}{-}\nabla^{\prime}\right)  =\theta_{\mathbb{J}^{D_{n+1}%
}(\pi)}^{\mathcal{J}^{n+1}(\pi)}\left(  \nabla\right)  \overset{\cdot}%
{-}\theta_{\mathbb{J}^{D_{n+1}}(\pi)}^{\mathcal{J}^{n+1}(\pi)}\left(
\nabla^{\prime}\right)
\]

\end{proposition}

Now we have

\begin{theorem}
\label{t6.6}The mappings $\theta_{\mathbb{J}^{D_{n}}(\pi)}^{\mathcal{J}%
^{n}(\pi)}:\mathcal{J}^{n}(\pi)\rightarrow\mathbb{J}^{D_{n}}(\pi)$\ are bijective.
\end{theorem}

\begin{proof}
The mapping $\theta_{\mathbb{J}^{D}(\pi)}^{\mathcal{J}^{1}(\pi)}%
:\mathcal{J}^{1}(\pi)\rightarrow\mathbb{J}^{D}(\pi)$\ is obviously bijective.
We proceed by induction on $n$. By Proposition \ref{t6.5} and the induction
hypothesis, $\left(  \theta_{\mathbb{J}^{D_{n+1}}(\pi)}^{\mathcal{J}^{n+1}%
(\pi)},\theta_{\mathbb{S}^{D_{n+1}}(\pi)}^{\mathcal{S}^{n+1}(\pi)}\underset
{E}{\times}\theta_{\mathbb{J}^{D_{n}}(\pi)}^{\mathcal{J}^{n}(\pi)}%
,\theta_{\mathbb{J}^{D_{n}}(\pi)}^{\mathcal{J}^{n}(\pi)}\right)  $\ gives a
morphism of affine bundles from the affine bundle $\pi_{n+1,n}:\mathcal{J}%
^{n+1}(\pi)\rightarrow\mathcal{J}^{n}(\pi)$ over the vector bundle
$\mathcal{S}^{n+1}(\pi)\underset{E}{\times}\mathcal{J}^{n}(\pi)\rightarrow
\mathcal{J}^{n}(\pi)$\ to the affine bundle $\pi_{n+1,n}:\mathbb{J}^{D_{n+1}%
}(\pi)\rightarrow\mathbb{J}^{D_{n}}(\pi)$\ over the vector bundle
$\mathbb{S}^{D_{n+1}}(\pi)\underset{E}{\times}\mathbb{J}^{D_{n}}%
(\pi)\rightarrow\mathbb{J}^{D_{n}}(\pi)$\ is an isomorphism of affine bundles,
so that the mapping $\theta_{\mathbb{J}^{D_{n+1}}(\pi)}^{\mathcal{J}^{n+1}%
(\pi)}:\mathcal{J}^{n+1}(\pi)\rightarrow\mathbb{J}^{D_{n+1}}(\pi)$\ is bijective.
\end{proof}

\begin{corollary}
The mappings $\psi_{n}:\mathbb{J}^{D^{n}}(\pi)\rightarrow\mathbb{J}^{D_{n}%
}(\pi)$\ are bijective.
\end{corollary}

\begin{proof}
This follows simply from Theorems \ref{t5.6}\ and \ref{t6.6}\ and the
commutativity of the following diagram:
\[%
\begin{array}
[c]{ccc}
& \mathcal{J}^{n}(\pi) & \\
\theta_{\mathbb{J}^{D^{n}}(\pi)}^{\mathcal{J}^{n}(\pi)}\swarrow &  &
\searrow\theta_{\mathbb{J}^{D_{n}}(\pi)}^{\mathcal{J}^{n}(\pi)}\\
\mathbb{J}^{D^{n}}(\pi) &
\begin{array}
[c]{c}%
\rightarrow\\
\psi_{n}%
\end{array}
& \mathbb{J}^{D_{n}}(\pi)
\end{array}
\]

\end{proof}

\end{document}